\documentclass[preprint,numbers,addressatend]{imsart}

\usepackage{amsthm,amsmath,natbib}
\usepackage{amsfonts,amssymb,graphicx}
\RequirePackage[dvips]{hyperref}
\usepackage{color}

\RequirePackage{hypernat}

\startlocaldefs
\numberwithin{equation}{section}
\newtheorem{thm}{Theorem}[section]
\newtheorem{prop}[thm]{Proposition}
\newtheorem{lem}[thm]{Lemma}

\newcommand{\boldC}{\Pi}

\usepackage{graphicx}
\usepackage{verbatim}

\endlocaldefs

\begin{document}

\begin{frontmatter}

\title{Tail behavior of solutions of linear recursions on trees}
\runtitle{Linear recursions on trees}

\author{\fnms{Mariana} \snm{Olvera-Cravioto}\ead[label=e1]{molvera@ieor.columbia.edu}}
\address{Department of Industrial Engineering \\ and Operations Research \\ Columbia University \\ New York, NY 10027 \\ \printead{e1}}
\affiliation{Columbia University}

\runauthor{M. Olvera-Cravioto}

\begin{abstract}
Consider the linear nonhomogeneous fixed point equation 
$$R\stackrel{\mathcal{D}}{=}\sum_{i=1}^NC_iR_i+Q,$$ 
where $(Q,N,C_1,\dots,C_N)$ is a random vector with $N\in\{0,1,2,3,\dots\}\cup\{\infty\},\,\{C_i\}_{i=1}^N\geq0,\,P(|Q|>0)>0$, and $\{R_i\}_{i=1}^N$ is a sequence of i.i.d. random variables independent of $(Q,N,C_1,\dots,C_N)$ having the same distribution as $R$. It is known that $R$ will have a heavy-tailed distribution under several different sets of assumptions on the vector $(Q,N,C_1,\dots,C_N)$. This paper investigates the settings where either $Z_N=\sum_{i=1}^NC_i$ or $Q$ are regularly varying with index $-\alpha<-1$ and $E\left[\sum_{i=1}^NC_i^\alpha\right]<1$. This work complements previous results showing that $P(R>t)\sim Ht^{-\alpha}$ provided there exists a solution $\alpha>0$ to the equation $E\left[\sum_{i=1}^N|C_i|^\alpha\right]=1$, and both $Q$ and $Z_N$ have lighter tails.  
\end{abstract}



\begin{date}
\date{\today}
\end{date}

\end{frontmatter}

\section{Introduction}

Motivated by the analysis of information ranking algorithms, this paper investigates the tail behavior of the solution to the stochastic fixed point equation
\begin{equation} \label{eq:IntroLinear} 
R \stackrel{\mathcal{D}}{=} \sum_{i=1}^N C_i R_i + Q,
\end{equation}
where $(Q,N, C_1,\dots, C_N)$ is a random vector with $N \in \mathbb{N} \cup \{\infty\}$, $\mathbb{N} = \{0, 1, 2, 3, \dots\}$, $\{C_i\}_{i =1}^N \geq 0$, $P(|Q| >0) > 0$, and $\{R_i\}_{i \in \mathbb{N}}$ is a sequence of i.i.d. random variables independent of $(Q, N, C_1,\dots, C_N)$ having the same distribution as $R$; the symbol $\stackrel{\mathcal{D}}{=}$ denotes equality in distribution. This stochastic fixed point equation recently appeared in the analysis of Google's PageRank algorithm, which computes the ranks of pages on the World Wide Web according to the recursion
\begin{equation} \label{eq:PageRank}
PR(p_i) = \frac{1-d}{n} + d \sum_{p_j \in M(p_i)} \frac{PR(p_j)}{L(p_j)},
\end{equation}
where $p_1, p_2,\dots, p_n$ are the pages under consideration, $M(p_i)$ is the set of pages that link to $p_i$, $L(p_j)$ is the number of outbound links on page $p_j$, $PR(p_j)$ is the PageRank of page $p_j$, and $n$ is the total number of pages. A first order stochastic approximation for the rank of a randomly chosen page is obtained by multiplying both sides of \eqref{eq:PageRank} by $n$ and considering the fixed point equation
$$R \stackrel{\mathcal{D}}{=} 1-d + d \sum_{i=1}^N \frac{R_j}{D_j},$$
where $\{D_j\}$ are i.i.d. random variables distributed according to the out-degree distribution of the web graph, $N$ is a random variable distributed according to the in-degree distribution, and $\{R_j\}$ are i.i.d random variables having the same distribution as $R$. This approach, first introduced in \cite{Volk_Litv_Dona_07}, can be thought of as approximating the web graph with a branching tree, a well known technique used in the analysis of random graphs (see, e.g., \cite{Hof_Hoo_Van_05} and the references therein). 

The fixed point equation \eqref{eq:IntroLinear} has been recently analyzed in \cite{Jel_Olv_10, Volk_Litv_10} for the special case of $Q, N, \{C_i\}$ nonnegative and mutually independent, with the $\{C_i\}$  i.i.d.; in \cite{Volk_Litv_10} the pair $(Q, N)$ was allowed to be dependent under stronger moment conditions. One of the results in these articles was that when the distribution of $N$ is heavy-tailed, in particular, regularly varying, the tail distribution of $R$ is proportional to that of $N$, i.e., 
$$P(R > x) \sim H P(N > x) \qquad \text{as} \quad x \to \infty,$$
where $f(x) \sim g(x)$ is used throughout the paper to denote $\lim_{x \to \infty} f(x)/g(x) = 1$. This indicates that highly ranked pages are those with very large in-degree. One way in which this could be modified is by choosing a different set of weights $\{C_i\}$ in \eqref{eq:IntroLinear}. The general setting of $(Q, N, C_1,\dots, C_N)$ arbitrarily dependent with the $\{C_i\}$ not necessarily independent and/or identically distributed allows a great level of flexibility in this respect. This setting is also consistent with the broader literature on weighted branching processes \cite{Rosler_93} and branching random walks \cite{Biggins_77}, which appear  in the probabilistic analysis of other algorithms as well \cite{Ros_Rus_01, Nei_Rus_04}, e.g. Quicksort algorithm \cite{Fill_Jan_01}. 

A very well known special case of equation \eqref{eq:IntroLinear} is obtained by setting $N \equiv 1$, since then it becomes the stochastic recurrence equation
$$R \stackrel{\mathcal{D}}{=} C R + Q, \qquad (C, Q) \text{ independent of } R.$$
The power law tail asymptotics of the solution $R$ to this equation were established in the classical work of Kesten \cite{Kesten_73} (in a multivariate setting), and were also derived through the use of implicit renewal theory by Goldie \cite{Goldie_91}. The approach from \cite{Goldie_91} was generalized in \cite{Jel_Olv_11b} to analyze \eqref{eq:IntroLinear} for real-valued weights $\{C_i\}$. The main assumption in \cite{Jel_Olv_11b} (and the corresponding $N \equiv 1$ versions of \cite{Kesten_73, Goldie_91}) is the existence of a solution $\alpha > 0$ to the equation $E\left[ \sum_{i=1}^N |C_i|^\alpha \right] = 1$ such that $E\left[ \sum_{i=1}^N |C_i|^\alpha \log|C_i| \right] > 0$, $E[|Q|^\alpha] < \infty$, and if $\alpha >1$, $E\left[ \sum_{i=1}^N |C_i| \right] < 1$, $E\left[ \left( \sum_{i=1}^N |C_i| \right)^\alpha \right] < \infty$, in which case
$$P(R > x) \sim H x^{-\alpha} \qquad \text{as} \quad x \to \infty,$$
for some constant $H \geq 0$. The work of \cite{Jel_Olv_10} already shows that if such $\alpha$ does not exist, then $P( R > x)$ can still be regularly varying if either the distribution of $N$ or $Q$ are regularly varying. When $N \equiv 1$, $(C, Q)$ are generally dependent, $C \geq 0$ a.s. and $Q$ is regularly varying, the tail equivalence of $P(R > x)$ and $P(Q > x)$ was shown in \cite{Grey_94}.  The main results in this paper, Theorem \ref{T.Main_N} and Theorem \ref{T.MainQ}, give the corresponding generalization of the results in \cite{Grey_94, Jel_Olv_10} to arbitrarily dependent $(Q, N , C_1, \dots, C_N)$. In particular, it is shown that if either $P\left( \sum_{i=1}^N C_i > x \right)$, or $P(Q > x)$, are regularly varying with index $-\alpha < -1$, and certain moment conditions are satisfied, then
$$P( R > x) \sim H' P\left( \sum_{i=1}^N C_i > x \right), \quad \text{respectively}, \quad P(R > x) \sim H'' P(Q > x)$$
as $x \to \infty$, for some explicit constants $H', H'' > 0$. We point out that \eqref{eq:IntroLinear} may also have light-tailed solutions, as the work in \cite{Goldie_Grubel_96} shows for the $N \equiv 1$ case, but we focus here only on the heavy-tailed ones. 
 
The paper is organized as follows. First we construct an explicit solution to \eqref{eq:IntroLinear} on a weighted branching tree. As will be discussed in more detail in Section \ref{S.Construction}, this particular solution is the only one of practical interest, since under mild technical conditions, this is the unique limit of  the process that results from the iteration of \eqref{eq:IntroLinear} (see Lemma \ref{L.Convergence}). The main result for the case where the tail behavior of $R$ is dominated by the sum of the weights, $\sum_{i=1}^N C_i$, is given in Section \ref{S.NDominates}, and the main result for the case where $Q$ dominates is given in Section \ref{S.QDominates}. The main technical contribution of the paper is in the derivation of uniform bounds (in $n$ and $x$) for the distribution of the sum of the weights in the $n$th generation of a weighted branching tree, $P(W_n > x)$, given in Propositions \ref{P.UniformBound} and \ref{P.UniformBoundQ}. These uniform bounds are the key tool in establishing the geometric rate of convergence of the iterations of the fixed point equation \eqref{eq:IntroLinear} to the solution $R$ constructed in Section \ref{S.Construction}. Finally, the more technical proofs are postponed to Section \ref{S.Proofs} and the Appendix. 

The last thing to mention is that the approach used to derive the uniform bounds from Propositions \ref{P.UniformBound} and \ref{P.UniformBoundQ} can also be helpful in the analysis of other recursions on trees, such as the ones studied in \cite{Jel_Olv_11a} and the more extensive survey of \cite{Aldo_Band_05}, e.g., 
\begin{equation*} 
R \stackrel{\mathcal{D}}{=} \left(\bigvee_{i=1}^N C_i R_i \right) \vee Q \qquad R \stackrel{\mathcal{D}}{=} \left(\bigvee_{i=1}^N C_i R_i \right) + Q,
\end{equation*}
that may fall outside of the implicit renewal theory framework of \cite{Jel_Olv_11a}.

\section{Construction of a solution on a tree} \label{S.Construction}

We start by constructing in this section a particular solution to the fixed point equation
\begin{equation} \label{eq:Linear}
R \stackrel{\mathcal{D}}{=} \sum_{i=1}^N C_i R_i + Q,
\end{equation}
where $(Q,N, C_1,\dots, C_N)$ is a random vector with $N \in \mathbb{N} \cup \{\infty\}$, $\{C_i\}_{i=1}^N \geq 0$, $P(|Q| > 0 ) > 0$, and $\{R_i\}_{i\in \mathbb{N}}$ is a sequence of i.i.d. real-valued random variables independent of $(Q, N, C_1,\dots, C_N)$ having the same distribution as $R$. We will show in Section \ref{SS.Iterations} that the process that results from iterating \eqref{eq:Linear} converges under mild conditions to this particular solution. 

First we construct a random tree $\mathcal{T}$. We use the notation $\emptyset$ to denote the root node of $\mathcal{T}$, and $A_n$, $n \geq 0$, to denote the set of all individuals in the $n$th generation of $\mathcal{T}$, $A_0 = \{\emptyset\}$. Let $Z_n$ be the number of individuals in the $n$th generation, that is, $Z_n = |A_n|$, where $| \cdot |$ denotes the cardinality of a set; in particular, $Z_0 = 1$. We iteratively construct the tree as follows. Let $N$ be the number of individuals born to the root node $\emptyset$, $N_\emptyset = N$, and let $\{N_{(i_1,\dots, i_n)} \}_{n \geq 1}$ be i.i.d. copies of $N$. Define now 
\begin{equation} \label{eq:AnDef}
A_1 = \{ i: 1 \leq i \leq N \}, \quad A_n = \{ (i_1, i_2, \dots, i_n): (i_1, \dots, i_{n-1}) \in A_{n-1}, 1 \leq i_n \leq N_{(i_1, \dots, i_{n-1})} \}.
\end{equation}
It follows that the number of individuals $Z_n = |A_n|$ in the $n$th generation, $n \geq 1$, satisfies the branching recursion 
$$Z_{n} = \sum_{(i_1, \dots, i_{n-1}) \in A_{n-1}} N_{(i_1,\dots, i_{n-1})}.$$ 

Next let $\mathbb{N}_+ = \{1, 2, 3, \dots\}$ be the set of positive integers and let $U = \bigcup_{k=0}^\infty (\mathbb{N}_+)^k$ be the set of all finite sequences ${\bf i} = (i_1, i_2, \dots, i_n) \in U$, where by convention $\mathbb{N}_+^0 = \{ \emptyset\}$ contains the null sequence $\emptyset$. Also, for ${\bf i} \in A_1$ we simply use the notation ${\bf i} = i_1$, that is, without the parenthesis. Similarly, for ${\bf i} = (i_1, \dots, i_n)$ we will use $({\bf i}, j) = (i_1,\dots, i_n, j)$ to denote the index concatenation operation, if ${\bf i} = \emptyset$, then $({\bf i}, j) = j$. 

Now, we construct the weighted branching tree $\mathcal{T}_{Q,C}$ as follows. The root node $\emptyset$ is assigned a vector $(Q_\emptyset, N_\emptyset, C_{(\emptyset, 1)}, \dots, C_{(\emptyset, N_\emptyset)}) = (Q, N, C_1, \dots, C_N)$ with $N \in \mathbb{N} \cup \{\infty\}$ and $P(|Q| > 0) > 0$; $N$ determines the number of nodes in the first generation of $\mathcal{T}$ according to \eqref{eq:AnDef}. Each node in the first generation is then assigned an i.i.d. copy $(Q_i, N_i, C_{(i,1)}, \dots, C_{(i,N_i)})$ of the root vector and the $\{N_i\}$ are used to define the second generation in $\mathcal{T}$ according to \eqref{eq:AnDef}. 
In general, for $n\ge 2$, to each node ${\bf i} \in A_{n-1}$, we assign an i.i.d. copy 
$(Q_{\bf i}, N_{\bf i}, C_{({\bf i},1)}, \dots, C_{({\bf i}, N_{\bf i})})$ of the root vector and construct 
$A_{n} = \{({\bf i}, i_{n}): {\bf i} \in A_{n-1}, 1 \leq i_{n} \leq N_{\bf i}\}$. 
Note that the vectors $(Q_{\bf i}, N_{\bf i}, C_{({\bf i},1)}, \dots, C_{({\bf i}, N_{\bf i})})$, ${\bf i} \in A_{n-1}$, are also chosen independently of all the 
previously assigned vectors  $(Q_{\bf j}, N_{\bf j}, C_{({\bf j},1)}, \dots, C_{({\bf j}, N_{\bf j})})$, ${\bf j} \in A_{k}, 0\le k\le n-2$.
For each node in $\mathcal{T}_{Q,C}$ we also define the weight $\boldC_{(i_1,\dots,i_n)}$ via the recursion
$$ \boldC_{i_1} =C_{i_1}, \qquad \boldC_{(i_1,\dots,i_n)} = C_{(i_1,\dots, i_n)} \boldC_{(i_1,\dots,i_{n-1})}, \quad n \geq 2,$$
where $\boldC =1$ is the weight of the root node. Note that the weight $\boldC_{(i_1,\dots, i_n)}$ is equal to the product of all the weights $C_{(\cdot)}$ along the branch leading to node $(i_1, \dots, i_n)$, as depicted in Figure \ref{F.Tree}.

\begin{center}
\begin{figure}[h,t]
\begin{picture}(430,160)(0,0)
\put(0,0){\includegraphics[scale = 0.8, bb = 0 510 500 700, clip]{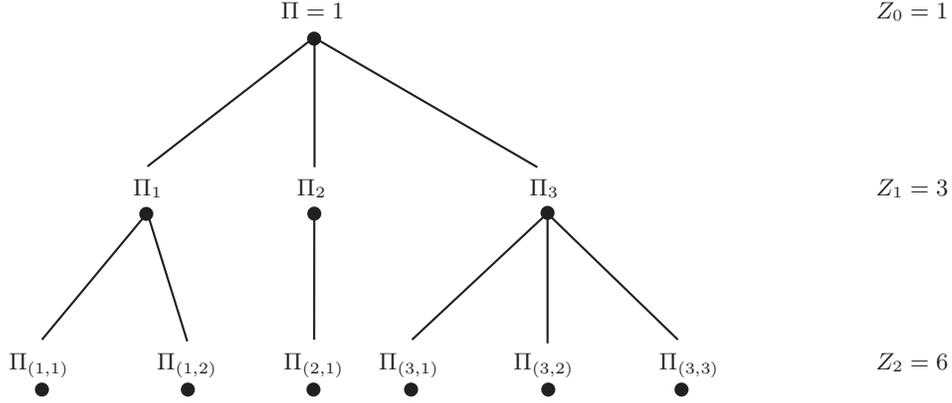}}
\put(125,150){\small $\boldC = 1$}
\put(69,83){\small $\boldC_{1}$}
\put(131,83){\small $\boldC_{2}$}
\put(219,83){\small $\boldC_{3}$}
\put(22,17){\small $\boldC_{(1,1)}$}
\put(78,17){\small $\boldC_{(1,2)}$}
\put(126,17){\small $\boldC_{(2,1)}$}
\put(162,17){\small $\boldC_{(3,1)}$}
\put(213,17){\small $\boldC_{(3,2)}$}
\put(268,17){\small $\boldC_{(3,3)}$}
\put(350,150){\small $Z_0 = 1$}
\put(350,83){\small $Z_1 = 3$}
\put(350,17){\small $Z_2 = 6$}
\end{picture}
\caption{Weighted branching tree}\label{F.Tree}
\end{figure}
\end{center}

We now define on the weighted branching tree $\mathcal{T}_{Q, C}$ the process
\begin{equation} \label{eq:W_k}
W_0 =  Q, \quad W_n =  \sum_{{\bf i} \in A_n} Q_{{\bf i}} \boldC_{{\bf i}}, \qquad n \geq 1,
\end{equation}
and the process $\{R^{(n)}\}_{n \geq 0}$ according to
\begin{equation} \label{eq:R_nDef}
R^{(n)} = \sum_{k=0}^n W_k , \qquad n \geq 0,
\end{equation}
that is, $R^{(n)}$ is the sum of the weights of all the nodes on the tree up to the $n$th generation. It is not hard to see that $R^{(n)}$ satisfies the recursion
\begin{equation} \label{eq:LinearRecSamplePath} 
R^{(n)} = \sum_{j=1}^{N_\emptyset} C_{(\emptyset,j)} R^{(n-1)}_{j} + Q_{\emptyset} = \sum_{j=1}^{N} C_{j} R^{(n-1)}_{j} + Q, \qquad n \geq 1,
\end{equation}
where $\{R_{j}^{(n-1)}\}$ are independent copies of $R^{(n-1)}$ corresponding to the tree starting with individual $j$ in the first generation and ending on the $n$th generation; note that $R_j^{(0)} = Q_j$. Moreover, since the tree structure repeats itself after the first generation, $W_n$ satisfies
\begin{align}
W_n &= \sum_{{\bf i} \in A_n} Q_{{\bf i}} \boldC_{{\bf i}} \notag\\
&= \sum_{k = 1}^{N_\emptyset} C_{(\emptyset,k)}  
\sum_{(k,\dots,i_n) \in A_n} Q_{(k,\dots,i_n)} \prod_{j=2}^n C_{(k,\dots,i_j)}  \notag\\
&\stackrel{\mathcal{D}}{=} \sum_{k=1}^N C_k W_{(n-1),k},\label{eq:WnRec}
\end{align}
where $\{W_{(n-1),k}\}$ is a sequence of i.i.d. random variables independent of $(N, C_1, \dots, C_N)$ and having the same distribution as $W_{n-1}$. 

The following result from \cite{Jel_Olv_11b} (Lemma 4.1) gives the convergence of $R^{(n)}$ to a proper limit. 

\begin{lem}
If for some $0 < \beta \leq 1$, $E[|Q|^\beta] < \infty$ and $E\left[ \sum_{j=1}^N C_j^\beta \right] < 1$, then $R^{(n)} \to R$ a.s. as $n \to \infty$, where $E[|R|^\beta] < \infty$ and is given by
\begin{equation} \label{eq:ExplicitConstr}
R \triangleq \sum_{n=0}^\infty W_n.
\end{equation}
\end{lem}

As discussed in \cite{Jel_Olv_11b}, the observation that the sum of all the absolute values of the weights on the tree are a.s. finite, i.e.,
$$\sum_{n=0}^\infty \sum_{{\bf i} \in A_n} |Q_{\bf i}| \Pi_{\bf i} < \infty \qquad \text{a.s.},$$
justifies the following identity
$$R = \sum_{j=1}^{N_\emptyset} C_{(\emptyset, j)} R_j^{(\infty)} + Q_\emptyset = \sum_{j=1}^N C_j R_j^{(\infty)} + Q,$$
where $\{R_j^{(\infty)} \}$ are independent copies of $R$ corresponding to the infinite subtree starting with individual $j$ in the first generation. This derivation provides in particular the existence of a solution in distribution to \eqref{eq:Linear}. 

The set of all solutions to \eqref{eq:Linear} was recently described in \cite{Alsm_Mein_10b} (see Theorem 2.3), where it was shown that all solutions can be obtained from the particular explicit solution $R$ given by \eqref{eq:ExplicitConstr} and a particular nonnegative solution to the fixed point equation
$$W \stackrel{\mathcal{D}}{=} \sum_{i=1}^N C_i^\alpha W_i$$
where $\alpha > 0$ solves $E\left[ \sum_{i=1}^N C_i^\alpha \right] = 1$.  Nonetheless, from an applications perspective, we are interested in the convergence of the process that results from iterating \eqref{eq:Linear}, and we will show that under mild moment conditions on the initial values this procedure always converges to $R$. Hence, the focus of this paper is only on the tail behavior of $R$ as defined by \eqref{eq:ExplicitConstr}. 

As for the solutions to the homogeneous linear equation ($Q \equiv 0$ in \eqref{eq:Linear}), we briefly mention that the set of solutions was fully described in \cite{Alsm_Bigg_Mein_10}, and the power law asymptotics of the particular solution constructed on the weighted branching tree, provided $E\left[ \sum_{i=1}^N C_i \right] = 1$, have been previously established in \cite{Liu_00} and \cite{Iksanov_04}. 

The remainder of the paper is organized as follows. In Section \ref{SS.MomentsLinear} we state moment bounds for $W_n$ and $R$. In Section \ref{SS.Iterations} we describe the process that results from iterating the fixed point equation \eqref{eq:Linear} and show that it converges in distribution to $R$. The main result for the case where the sum of the weights dominates the behavior of $R$ (the equivalent to the case where $N$ dominates in \cite{Jel_Olv_10}) is given in Section \ref{S.NDominates}; and the main result for the case where the behavior of $R$ is dominated by $Q$ is given in Section \ref{S.QDominates}. The proofs of the main results are given in Section \ref{S.Proofs} and some results for weighted random sums, that may be of independent interest, are given in the Appendix.

{\bf Notation:} Recall that throughout the paper the convention is to denote the random vector associated to the root node  $\emptyset$ by $(Q, N,C_1, \dots, C_N) \equiv (Q_\emptyset, N_\emptyset, C_{(\emptyset, 1)}, \dots, C_{(\emptyset, N_\emptyset)})$. We will also use 
$$\rho_\beta = E\left[ \sum_{i=1}^N C_i^\beta\right] \quad \text{for any $\beta > 0$, and} \quad \rho \equiv \rho_1.$$

\subsection{Moments of $W_n$ and $R$} \label{SS.MomentsLinear}

Let $A_{\mathcal{T}} = \bigcup_{n=0}^\infty A_n$ and note that
$$W_n^+ \leq \sum_{i \in A_n} Q_{\bf i}^+ \Pi_{\bf i}, \qquad n \geq 1,$$
$$\text{and} \quad R^+ \leq \sum_{n=0}^\infty W_n^+ \leq \sum_{{\bf i} \in A_{\mathcal{T}}} Q_{\bf i}^+ \Pi_{\bf i},$$
so Lemmas 4.2, 4.3 and 4.4 in \cite{Jel_Olv_11a} apply and we immediately obtain the following results; we use $x \vee y = \max\{x, y\}$.

\begin{prop} \label{P.GeneralMoments}
Assume $E[(Q^+)^\beta] < \infty$. Then,
\begin{enumerate}
\item if $0 < \beta \leq 1$, 
$$E[(W_n^+)^\beta] \leq E[(Q^+)^\beta] \rho_\beta^n,$$
\item if $\beta > 1$, $\rho \vee \rho_\beta < 1$, and $E\left[ \left( \sum_{i=1}^N C_i \right)^\beta \right] < \infty$, there exists a finite constant $K_\beta > 0$ such that
$$E[(W_n^+)^\beta] \leq K_\beta (\rho \vee \rho_\beta)^n,$$
\end{enumerate}
for all $n \geq 0$. 
\end{prop}

\begin{lem} \label{L.Moments_R}
Fix $\beta > 0$ and assume $E[ |Q|^\beta] < \infty$. In addition, suppose either (i) $\rho_\beta < 1$ for $0 < \beta < 1$, or (ii) $\rho \vee \rho_\beta < 1$ and $E\left[ \left( \sum_{i=1}^N C_i \right)^\beta \right] < \infty$ for $\beta \geq 1$. Then, $E[ |R|^\gamma] < \infty$ for all $0 < \gamma \leq \beta$. Moreover, if $\beta \geq 1$, $R^{(n)} \stackrel{L_\beta}{\to} R$, where $L_\beta$ stands for convergence in $(E| \cdot |^\beta)^{1/\beta}$ norm. 
\end{lem}

\subsection{Iterations of the fixed point equation} \label{SS.Iterations} 

In this section we describe the process that is obtained from iterating the fixed point equation \eqref{eq:Linear}, since this would be the natural approach to computing the ranks of pages in the World Wide Web example described in the introduction. Next we will show that under some technical conditions, this process converges in distribution to $R$. To this end, note that given an initial condition $R_0^*$ we can iterate \eqref{eq:Linear} to obtain the process 
\begin{equation} \label{eq:Rstar_Rec}
R_{n+1}^* = Q_n^* + \sum_{i=1}^{N_n^*} C_{n,i}^* R_{n,i}^*,
\end{equation}
where $\{ R_{n,i}^* \}_{i \in \mathbb{N}}$ are i.i.d. copies of $R_n^*$ from the previous iteration, independent of $(Q_{n}^*, N_n^*, C_{n,1}^*, \dots, C_{n,N_n}^*)$, and 
$\{Q_{n}^*,N_n^*, C_{n,1}^*,\dots, C_{n,N_n}^*\}_{n \in \mathbb{N}}$ is a vector sequence having the same distribution as the root node of the tree $(Q, N, C_1, \dots, C_N)$ with $N \in \mathbb{N} \cup \{\infty\}$, $\{C_i\}_{i=1}^N \geq 0$ and $P(|Q| > 0) > 0$.

Now, similarly as discussed in Section 2.2 in \cite{Jel_Olv_10}, it is not hard to show that $R_n^*$ admits the following representation
\begin{equation} \label{eq:connection}
R_n^* \stackrel{\mathcal{D}}{=} R^{(n-1)} + W_n(R_0^*),
\end{equation}
where
$$W_n(R_0^*) = \sum_{{\bf i} \in A_n} R_{0,{\bf i}}^* \boldC_{\bf i}$$
and $\{ R_{0, {\bf i}}^* \}$ are i.i.d. copies of $R_0^*$, independent of $\boldC_{\bf (\cdot)}$, $N_{\bf (\cdot)}$ and $Q_{\bf (\cdot)}$. The following lemma shows that $R_n^* \Rightarrow R$ for any initial condition $R_0^*$ satisfying a moment assumption, where $\Rightarrow$ denotes convergence in distribution.

\begin{lem} \label{L.Convergence}
For any $R_0^*$, if $E[|Q|^\beta], E[|R_0^*|^\beta] < \infty$ and $\rho_\beta < 1$ for some $0 < \beta \leq 1$, then
$$R_n^* \Rightarrow R,$$
with $E[|R|^\beta] < \infty$. Furthermore, under these assumptions, the distribution of $R$ is the unique solution with finite absolute $\beta$-moment to recursion \eqref{eq:Linear}. 
\end{lem}

\begin{proof}
In view of \eqref{eq:connection}, and since $R^{(n )}\to R$ a.s., the result will follow from Slutsky's Theorem (see Theorem~25.4 in \cite{Billingsley_1995}) once we show that $W_n(R_0^*) \Rightarrow 0$. To this end, recall that $W_n(R_0^*)$ is the same as $W_n$ if we substitute the $Q_{{\bf i}}$ by the $R_{0,{\bf i}}^*$. Then, for every $\epsilon > 0$ we have that
\begin{align*}
P( |W_n(R_0^*)| > \epsilon) &\leq \epsilon^{-\beta} E[ |W_n(R_0^*)|^\beta] \\
&\leq \epsilon^{-\beta} \rho_\beta^n E[|R_0^*|^\beta] ,
\end{align*}
where in the second inequality we applied Proposition \ref{P.GeneralMoments}(a) to both the positive and negative parts. Since by assumption the right-hand side converges to zero as $n \to \infty$, then $R_n^* \Rightarrow R$. Furthermore, $E[|R|^\beta] < \infty$ by Lemma \ref{L.Moments_R}. Clearly, under the assumptions, the distribution of $R$ represents the unique solution to \eqref{eq:Linear}, since any other possible solution with finite absolute $\beta$-moment would have to converge to the same limit.  
\end{proof}

\section{The case when the sum of the weights dominates} \label{S.NDominates}

As mentioned in the introduction, we are interested in analyzing the distributional properties of $R^{(n)}$ and $R$, in particular, when they have a heavy tail behavior. The work in \cite{Jel_Olv_11b} already describes one setting in which 
$$P(R > x) \sim H x^{-\alpha} \qquad \text{as } x \to \infty,$$
where $\alpha > 1$ is a solution to the equation $\rho_\alpha = 1$, $\rho < 1$, $E[|Q|^\alpha]$, and $E\left[ \left( \sum_{i=1}^N C_i \right)^\alpha \right] < \infty$ (plus some other technical conditions). Two other scenarios where $P(R > x)$ is heavy-tailed (in particular, regularly varying) are those when either 
$$Z_N \triangleq \sum_{i=1}^N C_i \qquad \text{or} \qquad Q$$
have regularly varying distributions.

Recall that a function $f$ is regularly varying at infinity with index $-\alpha$, denoted $f \in \mathcal{R}_{-\alpha}$, if $f(x) = x^{-\alpha} L(x)$ for some slowly varying function $L$; and $L: [0, \infty) \to (0, \infty)$ is slowly varying if $\lim_{x \to \infty} L(\lambda x)/L(x) = 1$ for any $\lambda > 0$. 

In this section we focus on the case where $P(Z_N > x) \in \mathcal{R}_{-\alpha}$ for some $\alpha > 1$, and $\rho \vee \rho_\alpha < 1$. The approach that we will follow is similar to that used in \cite{Jel_Olv_10} \S 5, except for the added complexity of allowing the vector $(Q, N, C_1, \dots, C_N)$ to be arbitrarily dependent, and the weights $\{C_i\}_{i=1}^N$ not necessarily identically distributed. We start by stating a lemma that describes the asymptotic behavior of $R^{(n)}$. The proof of this lemma is based on the use of some asymptotic limits for randomly stopped and randomly weighted sums recently developed in \cite{Olvera_11}, and adapted to be used in this setting in Theorem \ref{T.SumPlusQ_ZN} in the Appendix.  The main technical difficulty of extending this lemma to steady state ($R = R^{(\infty)}$) is to develop a uniform bound for $R-R^{(n)}$, which is enabled by the main technical result of the paper, Proposition \ref{P.UniformBound}.  The proof of Lemma~\ref{L.Finite_n} below can be  found in Section \ref{SS.Finite_n}.

\begin{lem} \label{L.Finite_n}
Let $Z_N = \sum_{i=1}^N C_i$ and suppose $\overline{G}(x) = P(Z_N > x) \in \mathcal{R}_{-\alpha}$ with  $\alpha > 1$, $E[|Q|^{\alpha+\epsilon}] < \infty$, $\rho_{\alpha+\epsilon} < \infty$ for some $\epsilon > 0$, $E[Q] > 0$, and $\rho < 1$. Then, for any fixed $n \in \{1, 2, 3,\dots\}$,
\begin{equation} \label{eq:Asymp_R_n}
P(R^{(n)} > x) \sim   \frac{(E[Q])^\alpha}{(1-\rho)^\alpha}  \sum_{k=0}^{n} \rho_\alpha^k (1-\rho^{n-k})^\alpha \, P( Z_N > x)
\end{equation}
as $x \to \infty$, where $R^{(n)}$ was defined in \eqref{eq:R_nDef}.
\end{lem}

\bigskip

From Lemma \ref{L.Finite_n} one can already guess that, provided $\rho \vee \rho_\alpha < 1$, the tail behavior of $R$ will be
$$P(R > x) \sim \frac{(E[Q])^\alpha}{(1-\rho)^\alpha} \sum_{k=0}^\infty \rho_\alpha^k P(Z_N > x)$$
as $x \to \infty$, assuming that the exchange of limits is justified. As mentioned above, this exchange represents the main technical difficulty in the paper (along with its counterpart for the case when $Q$ dominates the behavior of $R$, discussed in the next section). This result has already been proven in \cite{Jel_Olv_10} for the case where $Q, N, \{ C_i\}$ are all independent and the $\{C_i\}$ are i.i.d. using sample-path arguments, and in \cite{Volk_Litv_10} for the case where $(Q, N)$ is independent of $\{C_i\}$, the $\{C_i\}$ are i.i.d.,  using transform methods and Tauberian theorems. Here we follow the approach from \cite{Jel_Olv_10} where the main tool was a special case of the uniform bound given below. The proof of Proposition \ref{P.UniformBound} is given in Section \ref{SS.UniformBounds}.

\begin{prop} \label{P.UniformBound}
Let $Z_N = \sum_{i=1}^N C_i$ and suppose $\overline{G}(x) = P(Z_N > x) \in \mathcal{R}_{-\alpha}$ with  $\alpha > 1$. Assume further that $E[(Q^+)^{\alpha+\epsilon}] < \infty$ and $\rho_{\alpha+\epsilon} < \infty$ for some $\epsilon> 0$. Fix $\rho \vee \rho_\alpha < \eta < 1$. Then, there exists a finite constant $K = K(\eta,\epsilon) > 0$ such that for all $n \geq 1$ and all $x \geq 1$,
\begin{equation} \label{eq:UniformBound}
P(W_n^+ > x) \leq K  \eta^n P(Z_N > x).
\end{equation}
\end{prop}

{\sc Remark:} Note that we can easily obtain a weaker uniform bound by applying the moment estimate on $E[(W_n^+)^\beta]$ from Proposition \ref{P.GeneralMoments}, i.e., $P(W_n^+ > x) \leq E[(W_n^+)^\beta] x^{-\beta} \leq K_\beta (\rho \vee \rho_\beta)^n x^{-\beta}$ for some $0 < \beta < \alpha$, so the tradeoff in \eqref{eq:UniformBound} is a slightly larger geometric term for a lighter tail distribution.

\bigskip

Proposition \ref{P.UniformBound} is the key to establishing that $|R - R^{(n)}|$ goes to zero geometrically fast, which is more precisely stated in the following lemma. 

\begin{lem} \label{L.TailDifference}
Let $Z_N = \sum_{i=1}^N C_i$ and suppose $\overline{G}(x) = P(Z_N > x) \in \mathcal{R}_{-\alpha}$ with  $\alpha > 1$, $E[ |Q|^{\alpha+\epsilon}] < \infty$ and $\rho_{\alpha + \epsilon} < \infty$, for some $\epsilon > 0$, and $E[Q] > 0$.  Assume $\rho \vee \rho_\alpha < 1$, then, for any fixed $0 < \delta < 1$, $n_0 \in \{1,2,\dots\}$ and $\rho \vee \rho_\alpha < \eta < 1$, there exists a finite constant $K > 0$ that does not depend on $\delta$ or $n_0$ such that
$$\lim_{x \to \infty} \frac{P\left(  |R - R^{(n_0)}| > \delta x \right)}{\overline{G}(x)} \leq  \frac{K \eta^{n_0+1}}{\delta^{\alpha+1}  n_0}.$$ 
\end{lem}

\begin{proof}
Fix $\rho \vee \rho_\alpha < \eta_0 < \eta$ and $0 < r < \min\{ \alpha-1, 1 \}$. By Potter's Theorem (see Theorem 1.5.6 (iii) in \cite{BiGoTe1987}), there exists a constant $x_0 = x_0(2, r) \geq 1$ such that for all $x,y \geq x_0$,
\begin{equation} \label{eq:PotterZ_N}
\frac{\overline{G}(y)}{\overline{G}(x)} \leq 2 \max\left\{ (y/x)^{-\alpha + r}, \, (y/x)^{-\alpha-r} \right\}. 
\end{equation}

Now define $s_{n_0} = \sum_{n= n_0+1}^\infty n^{-2} \leq  \int_{n_0}^\infty t^{-2} dt = n_0^{-1}$ and $m(x;\delta, n_0) = \lfloor \sqrt{\delta x/(x_0 s_{n_0})} \rfloor$.  Then, the union bound gives
\begin{align*}
P\left( |R - R^{(n_0)}| > \delta x \right) &\leq P\left( \sum_{n=n_0+1}^\infty |W_n| > \delta x \right) \\
&\leq \sum_{n=n_0+1}^\infty P\left( |W_n| > \delta x n^{-2} / s_{n_0} \right) \\
&\leq \sum_{n = n_0+1}^{m(x;\delta, n_0)} K_0 \eta_0^n \overline{G}(\delta x n^{-2}/s_{n_0}) + \sum_{n=m(x;\delta, n_0) + 1}^\infty \frac{E[|W_n|^{\alpha-r}]}{ (\delta x n^{-2} / s_{n_0})^{\alpha-r} } \\
&\leq \sum_{n = n_0+1}^{m(x;\delta, n_0)} 2 K_0 \eta_0^n \left( \frac{\delta x n^{-2} / s_{n_0}}{x} \right)^{-\alpha- r} \overline{G}(x) + \sum_{n=m(x;\delta, n_0) + 1}^\infty \frac{K_{\alpha-r} (\rho \vee \rho_{\alpha-r})^n}{ (\delta x n^{-2} / s_{n_0})^{\alpha-r} } ,
\end{align*}
where in the third inequality we used Proposition \ref{P.UniformBound} (applied to both the positive and negative parts of $W_n$; $K_0 = K_0(\eta_0,\epsilon)$) and Markov's inequality, and in the fourth one we used \eqref{eq:PotterZ_N} and Proposition \ref{P.GeneralMoments}. It follows that
\begin{align*}
P\left( |R - R^{(n_0)}| > \delta x \right) &\leq \frac{2K_0 (s_{n_0})^{\alpha+r}}{\delta^{\alpha+r}} \sum_{n = n_0+1}^{m(x;\delta, n_0)}  \eta_0^n n^{2(\alpha+r)} \overline{G}( x) \\
&\hspace{5mm} + \frac{K_{\alpha-r} (s_{n_0})^{\alpha-r}}{\delta^{\alpha-r}} \sum_{n= m(x;\delta, n_0) + 1}^\infty \frac{(\rho \vee \rho_{\alpha-r})^n n^{2(\alpha-r)} }{ x^{\alpha-r} } .
\end{align*}
Now note that by the convexity of $f(\theta) = \rho_\theta$ we have $\rho \vee \rho_{\alpha-r} \leq \rho \vee \rho_\alpha < \eta_0$, from where it follows that
\begin{align*}
P\left( |R - R^{(n_0)}| > \delta x \right) &\leq \frac{(2K_0+K_{\alpha-r}) (s_{n_0})^{\alpha-r}}{\delta^{\alpha+r}} \left( \sum_{n = n_0+1}^{ \infty }  \eta_0^n n^{2(\alpha+r)} \overline{G}( x) \right. \\
&\hspace{5mm}  \left. + \sum_{n= m(x;\delta, n_0) + 1}^\infty \eta_0^n n^{2(\alpha+r)} x^{-\alpha+r}  \right) \\
&\leq \frac{2K_0 + K_{\alpha-r}}{\delta^{\alpha+r} n_0^{\alpha-r} (1-\eta)} \cdot \sup_{m \geq 1} (\eta_0/\eta)^m m^{2(\alpha+r)}  \\
&\hspace{5mm} \cdot \left( \sum_{n=n_0+1}^\infty (1-\eta) \eta^n \overline{G}( x) +  \sum_{n=m(x;\delta, n_0) + 1}^\infty (1-\eta) \eta^n x^{-\alpha+r} \right) \\
&\triangleq \frac{K}{\delta^{\alpha+r} n_0^{\alpha-r}} \left( \eta^{n_0+1} \overline{G}( x) + \eta^{m(x;\delta, n_0) + 1} x^{-\alpha+r} \right) ,
\end{align*}
where $K = K(\eta_0, \epsilon, \eta,r)$ does not depend on $\delta$ or $n_0$. It follows that for $\overline{G}(x) = x^{-\alpha} L(x)$, 
$$\lim_{x \to \infty} \frac{P\left( R - R^{(n_0)} > \delta x \right)}{\overline{G}(x)} \leq \frac{K \eta^{n_0+1}}{\delta^{\alpha+1} n_0} \left( 1+ \lim_{x \to \infty} \frac{\eta^{\sqrt{\delta x/(x_0 s_{n_0})}-n_0-1} x^{r}}{L(x)} \right) = \frac{K \eta^{n_0+1}}{\delta^{\alpha+1} n_0}.$$
\end{proof}

\bigskip

Having stated Lemma \ref{L.TailDifference}, we can now prove the main theorem of this section.

\begin{thm} \label{T.Main_N}
Let $Z_N = \sum_{i=1}^N C_i$ and suppose $\overline{G}(x) = P(Z_N > x) \in \mathcal{R}_{-\alpha}$ with  $\alpha > 1$, $E[ |Q|^{\alpha+\epsilon}] < \infty$ and $\rho_{\alpha+\epsilon} < \infty$, for some $\epsilon > 0$, and $E[Q] > 0$. Assume $\rho \vee \rho_\alpha < 1$, then, 
$$P(R > x) \sim \frac{(E[Q])^\alpha}{(1-\rho)^\alpha(1-\rho_\alpha)} P(Z_N > x)$$
as $x \to \infty$, where $R$ was defined in \eqref{eq:ExplicitConstr}.  
\end{thm}

{\sc Remarks:} (i) For the case where the $\{C_i\}$ are i.i.d. and independent of $N$, and $P(N > x) \in \mathcal{R}_{-\alpha}$, Lemma~3.7(2) in \cite{Jess_Miko_06} gives
$$P(Z_N > x) \sim (E[ C_1])^\alpha P(N > x) \qquad \text{as } x \to \infty.$$
(ii) Given the previous remark, it follows that Theorem \ref{T.Main_N} generalizes both Theorem 5.1 in \cite{Jel_Olv_10} (for $Q, N, \{C_i\}$ all independent and $\{C_i\}$ i.i.d.) and the corresponding result from Section 3.4 in \cite{Volk_Litv_10} (for $(Q, N)$ independent of $\{C_i\}$, $\{C_i\}$ i.i.d., $E[Q]< 1$ and $E[C] = (1-E[Q])/E[N]$).  (iii) In view of Lemma \ref{L.Finite_n}, the theorem shows that the limits $\lim_{x \to \infty} \lim_{n \to \infty} P(R^{(n)} > x)/ P(N > x)$ are interchangeable.

\bigskip

\begin{proof}[Proof of Theorem \ref{T.Main_N}]
Fix $0< \delta <1$ and $n_0 \geq 1$. Choose $\rho \vee \rho_\alpha < \eta < 1$ and use Proposition~\ref{P.UniformBound} to obtain that for some constant $K_0 > 0$, 
$$P(W_n^+ > x) \leq K_0 \eta^n P(Z_N > x)$$
for all $n \geq 1$ and all $x \geq 1$. Let $H_n = (E[Q])^\alpha (1-\rho)^{-\alpha} \sum_{k=0}^n \rho_\alpha^k (1-\rho^{n-k})^\alpha$ and $H = H_\infty$.  Then,
\begin{align}
\left| P(R > x) -   H P(Z_N > x) \right| 
&\leq \left| P(R > x) -  P(R^{(n_0)} > x) \right| \label{eq:Tail} \\
&\hspace{5mm} + \left| P(R^{(n_0)} > x) -   H_{n_0} P(Z_N > x) \right| \label{eq:FiniteIterations} \\
&\hspace{5mm} + \left| H_{n_0} - H \right| P(Z_N > x). \label{eq:FiniteError}
\end{align}
By Lemma \ref{L.Finite_n}, there exists a function $\varphi(x) \downarrow 0$ as $x \to \infty$ such that 
\begin{equation} \label{eq:FiniteAsymptotics}
 \left| P(R^{(n_0)} > x) -   H_{n_0} P(Z_N > x) \right| \leq \varphi(x) H P(Z_N > x),
 \end{equation}
which can be used to bound \eqref{eq:FiniteIterations}. Next, for \eqref{eq:FiniteError} simply note that
\begin{align*}
\frac{1}{H} \left| H_{n_0} - H \right| &= (1-\rho_\alpha) \left( \sum_{k=0}^\infty \rho_\alpha^k - \sum_{k=0}^{n_0} \rho_\alpha^k (1-\rho^{n_0-k})^\alpha  \right)  \\
&= (1-\rho_\alpha) \sum_{k=0}^{n_0} \rho_\alpha^k (1 - (1-\rho^{n_0-k})^\alpha ) + (1-\rho_\alpha) \sum_{k=n_0+1}^\infty \rho_\alpha^k \\
&\leq (1-\rho_\alpha) \sum_{k=0}^{n_0} \rho_\alpha^k \alpha \rho^{n_0-k}  + \rho_\alpha^{n_0+1} \\
&\leq \left( \alpha (1-\rho_\alpha) (n_0+1) + \rho_\alpha \right) (\rho_\alpha \vee \rho)^{n_0} \\
&\leq \left( \alpha  \sup_{m \geq 1} \left( \frac{\rho_\alpha \vee \rho}{\eta} \right)^m m  \right) \eta^{n_0} \triangleq K' \eta^{n_0} .
\end{align*}
The rest of the proof is basically an analysis of \eqref{eq:Tail}. We start by noting that
\begin{align*}
\left| P(R > x) -  P(R^{(n_0)} > x) \right| &\leq  \left| P\left(R> x, \, |R-R^{(n_0)}| \leq \delta x\right)  -  P\left( R^{(n_0)} > x \right) \right| \\
&\hspace{5mm} + P\left( R > x, \, |R-R^{(n_0)}| > \delta x\right) .
\end{align*}
Also, since
$$P\left(R^{(n_0)}  > (1+\delta) x, \, |R-R^{(n_0)}| \leq \delta x\right)  \leq P\left(R> x, \, |R-R^{(n_0)}| \leq \delta x\right)  \leq P\left(R^{(n_0)} > (1-\delta) x \right) $$
and if $\overline{a} \leq a \leq \underline{a}$, then $|a - b| \leq |\underline{a} - b| +  |\overline{a} - b|$, we have 
\begin{align}
\left| P(R > x) -  P(R^{(n_0)} > x) \right| &\leq \left| P\left(R^{(n_0)} > (1+\delta)x, \, |R - R^{(n_0)}| \leq \delta x\right) - P\left(R^{(n_0)} > x \right) \right| \notag \\
&\hspace{5mm} + \left| P\left( R^{(n_0)} > (1-\delta)x \right) - P\left( R^{(n_0)} > x \right) \right| \notag \\
&\hspace{5mm} + P\left( R > x, \,  |R-R^{(n_0)}| > \delta x\right) \notag \\
&= P\left( R^{(n_0)} > (1-\delta)x \right) - P\left(R^{(n_0)} > (1+\delta)x, \, |R - R^{(n_0)}| \leq \delta x\right) \notag  \\
&\hspace{5mm}  + P\left( R > x, \,  |R-R^{(n_0)}| > \delta x\right) \notag \\
&\leq P\left( R^{(n_0)} > (1-\delta)x \right) - P\left( R^{(n_0)} > (1+\delta)x \right) \label{eq:DeltaDifference} \\
&\hspace{5mm} +2 P\left( |R - R^{(n_0)}| > \delta x \right) . \label{eq:TailDifference}
\end{align}
From \eqref{eq:FiniteAsymptotics} and the observation that $H_{n_0} \leq H$, it follows that \eqref{eq:DeltaDifference} is bounded by
\begin{align*}
&P(R^{(n_0)} > (1-\delta) x) - H_{n_0} P(Z_N > (1-\delta) x) \\
&\hspace{5mm} + H_{n_0} \left( P(Z_N > (1-\delta) x) - P(Z_N > (1+\delta)x) \right) \\
&\hspace{5mm} + H_{n_0} P(Z_N > (1+\delta)x) -  P(R^{(n_0)} > (1+\delta) x)  \\
&\leq \left\{ 2\varphi((1-\delta)x) \frac{\overline{G}((1-\delta)x)}{\overline{G}(x)} + \left( \frac{\overline{G}((1-\delta)x)}{\overline{G}(x)} - \frac{\overline{G}((1+\delta)x)}{\overline{G}(x)} \right) \right\} H P(Z_N > x).
\end{align*}
Moreover, since $\overline{G} \in \mathcal{R}_{-\alpha}$ and $\varphi((1-\delta)x) \to 0$, then 
$$2\varphi((1-\delta)x) \frac{\overline{G}((1-\delta)x)}{\overline{G}(x)} + \left( \frac{\overline{G}((1-\delta)x)}{\overline{G}(x)} - \frac{\overline{G}((1+\delta)x)}{\overline{G}(x)} \right) \to (1-\delta)^{-\alpha} - (1+\delta)^{-\alpha}$$
as $x \to \infty$. To analyze \eqref{eq:TailDifference} use Lemma \ref{L.TailDifference} to obtain
$$\lim_{x \to \infty} \frac{2 P\left( |R - R^{(n_0)}| > \delta x \right)}{H P(Z_N > x)} \leq \frac{K'' \eta^{n_0+1}}{\delta^{\alpha+1} n_0}$$
for any $\rho \vee \rho_\alpha < \eta < 1$ and some constant $K'' > 0$ that does not depend on $\delta$ or $n_0$. 

Finally, by replacing the preceding estimates in \eqref{eq:Tail} - \eqref{eq:FiniteError}, we obtain
\begin{align*}
\lim_{x \to \infty} \left| \frac{P(R > x)}{H P(Z_N > x)} - 1 \right| &\leq (1-\delta)^{-\alpha} - (1+\delta)^{-\alpha} +  \frac{K \eta^{n_0}}{\delta^{\alpha+1}} .
\end{align*}
Since the right hand side can be made arbitrarily small by first letting $n_0 \to \infty$ and then $\delta \downarrow 0$, the result of the theorem follows. 
\end{proof}

\section{The case when $Q$ dominates} \label{S.QDominates}

This section of the paper treats the case when the heavy-tailed behavior of $R$ arises from the $\{Q_{\bf i}\}$, known in the autoregressive processes literature as innovations. This setting is well known in the special case  $N \equiv 1$, since then the linear fixed point equation \eqref{eq:Linear} reduces to 
$$R \stackrel{\mathcal{D}}{=} C R + Q,$$
where $(C, Q)$ are generally dependent. This fixed point equation is the one satisfied by the steady state of the autoregressive process of order one with random coefficients, RCA(1) (see \cite{Kesten_73, Brandt_86, Goldie_91, Grey_94}). The power law asymptotics of the solution $R$ in this context were established in the classical work by Kesten \cite{Kesten_73} (multivariate setting), and through implicit renewal theory in the paper by Goldie \cite{Goldie_91}. In both of these works the assumptions include the existence of an $\alpha > 0$ such that $E[|C|^\alpha] = 1$, $E[|C|^\alpha \log^+|C|] < \infty$, and $E[|Q|^\alpha] < \infty$.  

That the innovations $\{Q_{\bf i}\}$ can give raise to heavy tails when the $\alpha$ mentioned above does not exist  is also well known, see, e.g. \cite{Grey_94, Kon_Mik_05}; the main theorem of this section provides an alternative derivation of the forward implication in Theorem~1 from \cite{Grey_94} (see also Proposition 2.4 in \cite{Kon_Mik_05}) in the more general context of $N \geq 0$.  We also mention that Theorem 1 in \cite{Grey_94} includes the case where $\alpha \in (0, 1]$, which would require a different proof technique from the one in this paper.  

The results presented here are very similar to those in Section \ref{S.NDominates}, and so are their proofs. We will therefore only present the statements and skip most of the proofs. We start with the equivalent of Lemma~\ref{L.Finite_n} in this context; its proof can be found in Section \ref{SS.Finite_n}. 

\begin{lem} \label{L.Finite_nQ}
Suppose $P(Q > x) \in \mathcal{R}_{-\alpha}$, with $\alpha > 1$, $E[ (Q^-)^{1+\epsilon}] < \infty$ and $E[Z_N^{\alpha+\epsilon}] < \infty$, for some $\epsilon > 0$. Then, for any fixed $n \in \{1,2,3,\dots\}$,
$$P(R^{(n)} > x) \sim \sum_{k=0}^n \rho_\alpha^k \, P(Q > x)$$
as $x \to \infty$, where $R^{(n)}$ was defined in \eqref{eq:ExplicitConstr}. 
\end{lem}

As for the case when $Z_N = \sum_{i=1}^N C_i$ dominates the asymptotic behavior of $R$, we can expect that, 
$$P(R > x) \sim (1-\rho_\alpha)^{-1} P(Q > x),$$
and the technical difficulty is justifying the exchange of limits. The same techniques used in Section~\ref{S.NDominates} can be used in this case as well.  The corresponding version of Proposition \ref{P.UniformBound} is given below. We point out that even though the condition $\rho < 1$ is not necessary for the proportionality constant in Lemma \ref{L.Finite_nQ} to be finite, it is required for the finiteness of $E[|R|^\beta]$ for some $\beta \geq 1$, so it is natural that it appears as part of the hypothesis in all the other results in this section.

\begin{prop} \label{P.UniformBoundQ}
Suppose $P(Q > x) \in \mathcal{R}_{-\alpha}$, with $\alpha > 1$, $E[Z_N^{\alpha+\epsilon}] < \infty$ for some $\epsilon > 0$, and let $\rho \vee \rho_\alpha < \eta < 1$. Then, there exists a constant $K = K(\eta,\epsilon) > 0$ such that for all $n \geq 1$ and all $x \geq 1$,
$$P(W_n > x) \leq K  \eta^n P(Q > x).$$
\end{prop}

A sketch of the proof can be found in Section \ref{SS.UniformBounds}.

\bigskip

The corresponding version of Lemma \ref{L.TailDifference} is given below. Its proof is basically identical to that of Lemma~\ref{L.TailDifference} and is therefore omitted. 

\begin{lem} 
Let $Z_N = \sum_{i=1}^N C_i$ and suppose $P(Q > x) \in \mathcal{R}_{-\alpha}$ with  $\alpha > 1$, $E[ Z_N^{\alpha+\epsilon}] < \infty$ for some $\epsilon > 0$, and $E[|Q|^\beta] < \infty$ for all $0 < \beta < \alpha$.  Assume $\rho \vee \rho_\alpha < 1$, then, for any fixed $0 < \delta < 1$, $n_0 \in \{1,2,\dots\}$ and $\rho \vee \rho_\alpha < \eta < 1$, there exists a constant $K > 0$ that does not depend on $\delta$ or $n_0$ such that
$$\lim_{x \to \infty} \frac{P\left(  |R - R^{(n_0)}| > \delta x \right)}{P(Q > x)} \leq  \frac{K \eta^{n_0+1}}{\delta^{\alpha+1} n_0}.$$ 
\end{lem}

And finally, the main theorem of this section. The proof again greatly resembles that of Theorem~\ref{T.Main_N} and is therefore omitted. 

\begin{thm} \label{T.MainQ}
Suppose $P(Q > x) \in \mathcal{R}_{-\alpha}$, with $\alpha > 1$, $E[|Q|^\beta] < \infty$ for all $0 < \beta < \alpha$. Assume $\rho \vee \rho_\alpha < 1$, and $E[Z_N^{\alpha+\epsilon}] < \infty$ for some $\epsilon > 0$. Then, 
$$P(R > x) \sim (1-\rho_\alpha)^{-1} P(Q > x)$$
as $x \to \infty$, where $R$ was defined in \eqref{eq:ExplicitConstr}.  
\end{thm}

\section{Proofs}\label{S.Proofs}

\subsection{Finite iterations of $R^{(n)}$} \label{SS.Finite_n}

This section contains the proofs of Lemma \ref{L.Finite_n} and Lemma \ref{L.Finite_nQ}, which refer to the asymptotic behavior of $P(R^{(n)} > x)$ for any finite $n$. 

\begin{proof}[Proof of Lemma \ref{L.Finite_n}]
We proceed by induction in $n$. For $n = 1$ fix $\alpha/(\alpha+\epsilon) < \delta < 1$ and note that by Theorem \ref{T.SumPlusQ_ZN}
\begin{align*}
P(R^{(1)} > x) &= P\left( \sum_{i=1}^{N_\emptyset} C_{(\emptyset,i)} Q_{(\emptyset,i)} + Q_\emptyset > x \right) \\
&= P\left( \sum_{i=1}^N C_i Q_i > x - Q, \, |Q| \leq x^\delta \right) + P\left( \sum_{i=1}^N C_i Q_i > x - Q, \, |Q| >  x^\delta \right),
\end{align*}
where $Q$ is independent of the $\{Q_i\}$ but not of $(N, C_1,\dots, C_N)$.  By Theorem 2.6 in \cite{Olvera_11} and the regular variation of $\overline{G}$,
$$P\left( \sum_{i=1}^N C_i Q_i > x - Q, \, |Q| \leq x^\delta \right) \leq P\left( \sum_{i=1}^N C_i Q_i > x - x^\delta \right) \sim P\left(Z_N > (x-x^\delta)/E[Q] \right) \sim (E[Q])^\alpha \overline{G}(x),$$
as $x \to \infty$, and
\begin{align*}
P\left( \sum_{i=1}^N C_i Q_i > x - Q, \, |Q| \leq x^\delta \right) &\geq P\left( \sum_{i=1}^N C_i Q_i > x +x^\delta\right) - P\left( |Q| > x^\delta \right) \\
&= (E[Q])^\alpha \overline{G}(x) (1+o(1)) - P\left( |Q| > x^\delta \right).
\end{align*}

Now note that by Markov's inequality, 
\begin{align*}
P\left( |Q| >  x^\delta \right) &\leq \frac{E[|Q|^{\alpha+\epsilon}]}{x^{\delta(\alpha+\epsilon)}} = o\left( \overline{G}(x) \right)
\end{align*}
as $x \to \infty$. Therefore,
$$P(R^{(1)} > x) \sim (E[Q])^\alpha \overline{G}(x).$$
Now suppose that we have
$$P(R^{(n)} > x) \sim \frac{(E[Q])^\alpha}{(1-\rho)^\alpha}  \sum_{k=0}^{n} \rho_\alpha^k (1-\rho^{n-k})^\alpha \, \overline{G}(x) \triangleq H_n \overline{G}(x).$$
By Theorem \ref{T.SumPlusQ_ZN}, 
\begin{align*}
P( R^{(n+1)} > x) &= P\left( \sum_{i=1}^N C_i R_i^{(n)} + Q > x \right) \\
&\sim \left( \rho_\alpha + H_n^{-1} (E[R^{(n)}])^\alpha \right) P(R^{(n)} > x) \\
&\sim \left( H_n \rho_\alpha + (E[R^{(n)}])^\alpha \right) \overline{G}(x) . 
\end{align*}
Next, observing that $E[R^{(n)}] = \sum_{i=0}^n E[W_i] = E[Q] \sum_{i=0}^n \rho^i = E[Q] (1-\rho^{n+1})/(1-\rho)$, we obtain
\begin{align*}
H_n \rho_\alpha + (E[R^{(n)}])^\alpha &= H_n \rho_\alpha + (E[Q])^\alpha \frac{(1-\rho^{n+1})^\alpha}{(1-\rho)^\alpha} \\
&=  \frac{(E[Q])^\alpha}{(1-\rho)^\alpha}  \sum_{k=0}^{n} \rho_\alpha^{k+1} (1-\rho^{n-k})^\alpha + \frac{(E[Q])^\alpha}{(1-\rho)^\alpha} (1-\rho^{n+1})^\alpha \\
&= \frac{(E[Q])^\alpha}{(1-\rho)^\alpha}  \sum_{j=1}^{n+1} \rho_\alpha^{j} (1-\rho^{n+1-j})^\alpha + \frac{(E[Q])^\alpha}{(1-\rho)^\alpha} (1-\rho^{n+1})^\alpha = H_{n+1}.
\end{align*}
This completes the proof. 
\end{proof} 

\bigskip

\begin{proof}[Proof of Lemma \ref{L.Finite_nQ}]
We proceed by induction in $n$. By Theorem \ref{T.SumPlusQ_Q},
$$P(R^{(1)} > x) = P\left( \sum_{i=1}^{N_\emptyset} C_{(\emptyset, i)} Q_{(\emptyset,i)} + Q _\emptyset > x \right) \sim \left( \rho_\alpha + 1 \right) P(Q > x)$$
as $x \to \infty$.  

Now suppose that we have
$$P(R^{(n)} > x) \sim \sum_{k=0}^n \rho_\alpha^k \, P(Q > x) \triangleq H_n P(Q > x).$$
Then by Theorem \ref{T.SumPlusQ_Q} again
\begin{align*}
P(R^{(n+1)} > x) &= P\left( \sum_{i=1}^{N_\emptyset} C_{(\emptyset, i)} R_{i}^{(n)} + Q_\emptyset > x \right) \\
&\sim \left( \rho_\alpha + H_n^{-1} \right) P(R^{(n)} > x) \\
&\sim \left( \rho_\alpha H_n + 1  \right) P(Q > x) \\
&= H_{n+1} P(Q > x) .
\end{align*}
\end{proof}

\subsection{Uniform bounds for $P(W_n > x)$} \label{SS.UniformBounds}

This section contains the proof of Proposition \ref{P.UniformBound} and a sketch of the proof of Proposition \ref{P.UniformBoundQ}. The first proof is rather involved, and a great effort goes into obtaining a bound for one iteration of the recursion satisfied by $W_n$, so for the convenience of the reader it is presented separately in Lemma \ref{L.Bound1Iter}.  This lemma can also be used to prove the corresponding uniform bound for $W_n$ in the case when $Q$ dominates the behavior of $R$. Throughout this section assume that $1/L(x)$ is locally bounded on $[1, \infty)$, and recall that if $L(t)$ is slowly varying, then $\lim_{t \to \infty} t^\varepsilon L(t) = \infty$ for any $\varepsilon > 0$.

\begin{lem} \label{L.Bound1Iter}
Suppose that $P(Z_N > x) \leq x^{-\alpha} L(x)$, with $\alpha > 1$ and $L(\cdot)$ slowly varying, $\rho \vee \rho_\alpha < 1$, and $\rho_{\alpha+\epsilon} < \infty$ for some $\epsilon > 0$. Assume further that $E[ (Q^+)^\beta] < \infty$ for any $0 < \beta < \alpha$.  Then, for any $0 < \delta < \min\{(\alpha-1)/2, \epsilon, 1/2\}$ and any $T > 0$,  there exists a finite constant $K = K(\epsilon, \delta, T) > 0$ that does not depend on $n$, such that for all $0 \leq n \leq \frac{\epsilon}{2|\log(\rho \vee \rho_\alpha)|} \log x$ and all $x \geq 1$,
\begin{equation*} 
P(W_{n+1}^+ > x) \leq K (\rho \vee \rho_\alpha)^n x^{-\alpha} L(x) + E\left[ 1\left( \sup_{1\leq i < N+1} C_i \leq x/T \right) \sum_{i=1}^N 1\left( C_i W_{n,i}^+ > (1-\delta)x  \right)  \right] .
\end{equation*}
\end{lem}

To ease the reading of the proof of Lemma \ref{L.Bound1Iter} we will split it into several lemmas. To avoid repetition we give below most of the definitions that will be used. We start by defining
\begin{align*}
I_N(t) &= \#\{ 1 \leq i < N+1: C_i > t\} \\
J_N(t) &= \#\{ 1 \leq i < N+1: C_i W_{n,i}^+ > t\}.
\end{align*}
For the same $\epsilon > 0$ and $0 < \delta < \min\{(\alpha-1)/2, \epsilon, 1/2\}$ from the statement of the lemma, define $\gamma_n = || W_n^+||_{1+\delta} = (E[(W_n^+)^{1+\delta}])^{1/(1+\delta)}$. We also define  
$$\nu = \epsilon/(2(\alpha+\epsilon)),  \qquad y = x/\log x, \qquad w = x^{1-\nu}, \qquad a_n = \delta^{-2} E[Q^+] (\rho \vee \rho_{1+\delta})^{n/(1+\delta)}.$$

Before going into the proof, we would like to emphasize that special care goes into making sure that $K$ in the statement of the lemma does not depend on $n$. This is important since Lemma~\ref{L.Bound1Iter} will be applied iteratively in the proof of Proposition \ref{P.UniformBound}, where one does not want $K$ to grow from one iteration to the next.

\begin{lem} \label{L.TruncSum}
Under the assumptions of Lemma \ref{L.Bound1Iter}, there exists a constant $K_1 = K_1(\delta,\nu, c)$ such that for all $x \geq 1$ and all $0 \leq n \leq c \log x$,
$$P\left(\sum_{i=1}^N C_i W_{n,i}^+ 1(C_i W_{n,i}^+ \leq y)  > \delta x, \, Z_N \leq x/a_n, \, I_N(w/\gamma_n) = 0 \right) \leq K_1 (\rho \vee \rho_\alpha)^n x^{-\alpha} L(x).$$
\end{lem}

\begin{proof}
By Lemma 3.2 in \cite{Olvera_11} (with $v = y$, $u = w$, $z = \delta x$, $\eta = 1+\delta$, and $A = (-\infty, x/a_n]$), there exists a constant $K_{1,1} = K_{1,1}(\delta) > 1$ such that \eqref{eq:TruncSum} is bounded by
\begin{align*}
&E\left[ 1(Z_N \leq x/a_n) e^{-\frac{\delta}{y} \log(y/w) \left(\delta x - \left( E[W_n^+] + \frac{K_{1,1} || W_n^+ ||_{1+\delta}}{\log (y/w)} \right)^+ Z_N \right)} \right] \\
&\leq e^{-\frac{\delta}{y} \log(y/w) \left(\delta x - \left( E[W_n^+] + \frac{K_{1,1} || W_n ||_{1+\delta}}{\log (y/w)} \right)^+ x/a_n \right)} \\
&\leq e^{-\frac{\delta}{y} \log(y/w) \left(\delta x - \left( E[Q^+] \rho^n + \frac{K_{1,1} \gamma_n }{\log (y/w)} \right) x/a_n \right)},
\end{align*}
where we used $E[W_n^+] \leq E[Q^+] \rho^n$ and $\gamma_n = ||W_n^+||_{1+\delta}$. From Proposition \ref{P.GeneralMoments} we know that $||W_n^+||_{1+\delta} = \left( E[(W_n^+)^{1+\delta}] \right)^{1/(1+\delta)} \leq K_{1,2} (\rho \vee \rho_{1+\delta})^{n/(1+\delta)} \leq K_{1,2} (\delta^2/E[Q^+]) a_n$, where $K_{1,2} = K_{1,2}(\delta) > 0$ is a finite constant. It follows that \eqref{eq:TruncSum} is bounded by
\begin{align}
e^{-\frac{\delta}{y} \log(y/w) \left(\delta x - \left( \delta^2 (\rho \vee \rho_{1+\delta})^{n\delta/(1+\delta)} + \frac{K_{1,1} K_{1,2} \delta^2 }{E[Q^+] \log (y/w)} \right) x \right)} &\leq e^{-\delta^2 \log x \log(y/w) \left(1 - \delta - \frac{K_{1,1} K_{1,2} \delta}{E[Q^+] \log (y/w)} \right)} \notag  \\
&= e^{-\delta^2\nu (\log x)^2 \left(1 - \frac{\log \log x}{\nu \log x} \right) \left(1 - \delta - \frac{K_{1,1} K_{1,2} \delta}{E[Q^+] \log (x^\nu/\log x)} \right)} \notag \\
&\leq e^{-\delta^2 \nu (1-2\delta)^2 (\log x)^2}, \label{eq:ExpBound}
\end{align}
where the last inequality holds for all $x \geq x_1$ for some $x_1 = x_1(\delta, \nu) > 0$. Now we choose $x_2 = x_2(\delta,\nu,c) \geq x_1$ such that $\delta^2 \nu(1-2\delta)^2 \log x \geq \alpha+\delta + c|\log (\rho \vee \rho_\alpha)|$ for all $x \geq x_2$ to obtain that \eqref{eq:ExpBound} is bounded by
\begin{equation*}
e^{-(\alpha+\delta) \log x - c|\log(\rho \vee \rho_\alpha)| \log x} = \frac{1}{x^{\alpha+\delta}} (\rho \vee \rho_\alpha)^{c\log x},
\end{equation*}
for all $x \geq x_2$. Next, define $K_{1,3} = K_{1,3}(\delta, \nu, c)$ as
$$K_{1,3} = \sup_{1 \leq t \leq x_2} \frac{1}{t^{-\alpha-\delta} (\rho \vee \rho_\alpha)^{c\log t}} \cdot e^{-\delta^2\nu (\log t)^2 \left(1 - \frac{\log \log t}{\nu \log t} \right) \left(1 - \delta - \frac{K_{1,1} K_{1,2} \delta}{E[Q^+] \log (t^\nu/\log t)} \right)} < \infty,$$
to obtain that 
$$P\left(\sum_{i=1}^N C_i W_{n,i}^+ 1(C_i W_{n,i}^+ \leq y)  > \delta x, \, Z_N \leq x/a_n, \, I_N(w/\gamma_n) = 0 \right) \leq  \frac{K_{1,3}}{x^{\alpha+\delta}} (\rho \vee \rho_\alpha)^{c\log x} \leq \frac{K_{1,3}}{x^{\alpha+\delta}} (\rho \vee \rho_\alpha)^{n}$$
for all $x \geq 1$ and $0 \leq n \leq c\log x$.  Finally, set $K_1 = K_{1,3} \sup_{t \geq 1} (t^\delta L(t))^{-1}$ to complete the proof. 
\end{proof}

\bigskip

\begin{lem} \label{L.SecondMax}
Under the assumptions of Lemma \ref{L.Bound1Iter}, there exists a finite constant $K_2 = K_2(\delta, \nu)$ such that for all $x \geq 1$,
$$P\left( J_N(y) \geq 2, \, Z_N \leq x/a_n, \, I_N(w/\gamma_n) = 0 \right) \leq K_2  (\rho \vee \rho_\alpha)^n x^{-\alpha} L(x).$$
\end{lem}

\begin{proof}
Let $\mathcal{F} = \sigma(N , C_1, \dots, C_N)$ and use the union bound to obtain
\begin{align*}
&P\left( J_N(y) \geq 2, \, Z_N \leq x/a_n, \, I_N(w/\gamma_n) = 0 \right) \\
&= E\left[ 1(Z_N \leq x/a_n, \, I_N(w/\gamma_n) = 0) E\left[ \left. 1\left( \bigcup_{1\leq i < j < N+1} \{ C_i W_{n,i}^+ > y, C_j W_{n,j}^+ > y\}  \right) \right| \mathcal{F} \right] \right] \\
&\leq E\left[ 1(Z_N \leq x/a_n, \, I_N(w/\gamma_n) = 0) \sum_{1\leq i< j < N+1}E\left[ \left. 1\left( C_i W_{n,i}^+ > y, C_j W_{n,j}^+ > y \right) \right| \mathcal{F} \right] \right] \\
&\leq E\left[ 1(Z_N \leq x/a_n, \, I_N(w/\gamma_n) = 0) \left( \sum_{i=1}^N E\left[  \left. 1\left( C_i W_{n,i}^+ > y \right) \right| C_i \right] \right)^2  \right],
\end{align*} 
where in the last step we used the conditional independence of $C_i W_{n,i}^+$ and $C_j W_{n,j}^+$ given $\mathcal{F}$. Now, by Markov's inequality,
\begin{align*}
1(I_N(w/\gamma_n) = 0) \sum_{i=1}^N E\left[  1\left( C_i W_{n,i}^+ > y \right) | C_i \right]  &\leq 1(I_N(w/\gamma_n) = 0)  \sum_{i=1}^N \frac{E[ (C_i W_{n,i}^+)^{1+\delta} | C_i]}{y^{1+\delta}} \\
&= \frac{1}{y^{1+\delta}} 1(I_N(w/\gamma_n) = 0)  \sum_{i=1}^N C_i^{1+\delta} \gamma_n^{1+\delta} \\
&\leq \frac{1}{y^{1+\delta}} (w/\gamma_n)^\delta \gamma_n^{1+\delta} Z_N = \frac{w^\delta}{y^{1+\delta}} \, \gamma_n Z_N.
\end{align*}
Similarly, for $\beta = \alpha-\delta\nu/2 > 1$, 
$$ \sum_{i=1}^N E\left[  1\left( C_i W_{n,i}^+ > y \right) | C_i \right] \leq  \sum_{i=1}^N \frac{E[ (C_i W_{n,i}^+)^\beta | C_i]}{y^\beta}  \leq \frac{E[ (W_n^+)^\beta]}{y^\beta} \sum_{i=1}^N C_i^\beta.$$
It follows that
\begin{align}
P\left( J_N(y) \geq 2, \, Z_N \leq x/a_n, \, I_N(w/\gamma_n) = 0 \right)  &\leq E\left[ 1(Z_N \leq x/a_n)  Z_N \sum_{i=1}^N C_i^\beta \right]   \frac{w^\delta}{y^{\beta + 1+\delta}} \, \gamma_n E[ (W_n^+)^\beta]  \notag \\
&\leq E\left[ \sum_{i=1}^N C_i^\beta \right]   \frac{w^\delta x}{y^{\beta + 1+\delta}} \, \frac{\gamma_n}{a_n} E[ (W_n^+)^\beta] \notag \\
&= \rho_\beta  \frac{(\log x)^{\beta+1+\delta} }{x^{\alpha +\delta\nu/2}} \, \frac{\gamma_n}{a_n} E[ (W_n^+)^\beta] . \label{eq:BoundSecondMax}
\end{align}
By Proposition \ref{P.GeneralMoments}, there exists a constant $K_{2,1} = K_{2,1}(\beta) > 0$ such that $E[(W_n^+)^\beta] \leq K_{2,1} (\rho \vee \rho_\beta)^n$. This combined with the observation $\gamma_n \leq K_{1,2} (\delta^2/E[Q^+]) a_n$ for some constant $K_{1,2} = K_{1,2}(\delta)$ made in the proof of Lemma \ref{L.TruncSum} gives that \eqref{eq:BoundSecondMax} is bounded by
\begin{equation*} 
\left( \rho_\beta \frac{K_{1,2} K_{2,1} \delta^2}{E[Q^+]}  \frac{(\log x)^{\beta+1+\delta}}{x^{\delta\nu/2} L(x)} \right) (\rho \vee \rho_\beta)^n x^{-\alpha} L(x) \leq  \left(\rho_\beta \frac{K_{1,2} K_{2,1} \delta^2}{E[Q^+]}   \sup_{t \geq 1} \frac{(\log t)^{\beta+1+\delta}}{t^{\nu\delta/2} L(t)}  \right)  (\rho \vee \rho_\beta)^n x^{-\alpha} L(x)
\end{equation*}
for all $x \geq 1$. The convexity of $f(\theta) = \rho_\theta$ gives $\rho \vee \rho_\beta \leq \rho \vee \rho_\alpha$, which completes the proof. 
\end{proof}

\bigskip

\begin{lem} \label{L.Weights}
Under the assumptions of Lemma \ref{L.Bound1Iter}, there exists a finite constant $K_3 = K_3(\delta,\epsilon, T)$ such that for all $x \geq 1$,
$$P\left( I_N(w/\gamma_n) \geq 1 \right) + P\left( I_N(w/\gamma_n) = 0, \, I_N(x/T) \geq 1 \right) \leq K_{3} \left( (\rho \vee \rho_\alpha)^n   + \frac{1}{x^{\epsilon/2}} \right) x^{-\alpha} L(x).$$
\end{lem}

\begin{proof}
First note that from the union bound we obtain
\begin{align}
P(I_N(w/\gamma_n) \geq 1) &= E \left[ 1\left( \bigcup_{i=1}^N \{ C_i > w/\gamma_n \} \right) \right] \leq E\left[ \sum_{i=1}^N 1(C_i > w/\gamma_n) \right] \notag \\
&= \sum_{i=1}^\infty P(C_i > w/\gamma_n, N \geq i) = \sum_{i=1}^\infty E[ 1(N \geq i) E[ 1(C_i > w/\gamma_n) |N] ] \notag \\
&\leq \frac{\gamma_n^{\alpha+\epsilon}}{w^{\alpha+\epsilon}} \sum_{i=1}^\infty E\left[ C_i^{\alpha+\epsilon} 1(N \geq i) \right] = E\left[ \sum_{i=1}^N C_i^{\alpha+\epsilon} \right] \frac{\gamma_n^{\alpha+\epsilon}}{x^{\alpha + \epsilon/2}}. \label{eq:WeightsBound1}
\end{align}
Recall from the proof of Lemma \ref{L.TruncSum} that $\gamma_n \leq K_{1,2} (\rho \vee \rho_{1+\delta})^{n/(1+\delta)}$ for some constant $K_{1,2} = K_{1,2}(\delta)$. It follows that \eqref{eq:WeightsBound1} is bounded by
$$\rho_{\alpha+\epsilon} \left( K_{1,2} ( \rho \vee \rho_{1+\delta})^{n/(1+\delta)} \right)^{\alpha+\epsilon} \frac{1}{x^{\alpha + \epsilon/2}} \leq \rho_{\alpha+\epsilon} K_{1,2}^{\alpha+\epsilon} (\rho \vee \rho_{1+\delta})^n \frac{1}{x^{\alpha+\epsilon/2}},$$
for any $x > 0$. Now, to bound the second probability in the statement, note that the same arguments used above give
\begin{align*}
P\left( I_N(w/\gamma_n) = 0, \, I_N(x/T) \geq 1   \right) &\leq P\left( I_N(x/T) \geq 1\right)  \leq \rho_{\alpha+\epsilon} T^{\alpha+\epsilon} \frac{1}{x^{\alpha+\epsilon}} .
\end{align*}
The convexity of $f(\theta) = \rho_\theta$ gives $\rho \vee \rho_{1+\delta} \leq \rho \vee \rho_\alpha$, from where it follows that there exists a constant $K_{3,1} = K_{3,1}(\delta, \epsilon, T)$ such that
\begin{align*}
P\left( I_N(w/\gamma_n) \geq 1 \right) + P\left( I_N(w/\gamma_n) = 0, \, I_N(x/T) \geq 1 \right) &\leq K_{3,1} \left( (\rho \vee \rho_\alpha)^n   + \frac{1}{x^{\epsilon/2}} \right) x^{-\alpha-\epsilon/2} \\
&\leq K_{3,1} \left( (\rho \vee \rho_\alpha)^n   + \frac{1}{x^{\epsilon/2}} \right) \sup_{t \geq 1} \frac{1}{t^{\epsilon/2}L(t)} \cdot x^{-\alpha} L(x)
\end{align*}
for all $x \geq 1$. 
\end{proof}

\bigskip

\begin{lem} \label{L.Potter}
Under the assumptions of Lemma \ref{L.Bound1Iter}, there exists a finite constant $K_4 = K_4(\delta)$ such that for all $x \geq 1$,
$$P\left( Z_N > x/a_n \right)  \leq K_{4} (\rho \vee \rho_\alpha)^n   x^{-\alpha} L(x).$$
\end{lem}

\begin{proof}
We use Potter's Theorem (see Theorem 1.5.6 (iii) in \cite{BiGoTe1987}) to obtain that there exists a constant $x_0 = x_0(2,\delta) > 0$ such that for all $\min\{x, x/a_0\} \geq x_0$
\begin{align}
P(Z_N > x/a_n) &\leq \frac{(x/a_n)^{-\alpha} L(x/a_n)}{x^{-\alpha} L(x)} \cdot x^{-\alpha} L(x) \notag \\
&\leq 2 \max\left\{ \left( \frac{x/a_n}{x} \right)^{-\alpha+\delta}, \, \left( \frac{x/a_n}{x} \right)^{-\alpha-\delta} \right\} x^{-\alpha} L(x) \notag \\
&= 2 \max\left\{ \left( \frac{E[Q^+] (\rho \vee \rho_{1+\delta})^{n/(1+\delta)} }{\delta^2} \right)^{\alpha-\delta}, \, \left( \frac{E[Q^+] (\rho \vee \rho_{1+\delta})^{n/(1+\delta)} }{\delta^2} \right)^{\alpha+\delta}  \right\} x^{-\alpha} L(x) \notag \\
&\leq \frac{2(E[Q^+] \vee 1)^{\alpha+\delta}}{\delta^{2(\alpha+\delta)}} (\rho \vee \rho_{1+\delta})^{(\alpha-\delta) n/(1+\delta)} x^{-\alpha} L(x). \notag
\end{align}
The convexity of $f(\theta) = \rho_\theta$ and our choice of $\delta$ gives $(\rho \vee \rho_{1+\delta})^{(\alpha-\delta)n/(1+\delta)} \leq (\rho \vee \rho_\alpha)^n$, from where it follows that
\begin{equation*}
P(Z_N > x/a_n) \leq \frac{2(E[Q^+] \vee 1)^{\alpha+\delta}}{\delta^{2(\alpha+\delta)}} (\rho \vee \rho_\alpha)^{ n} x^{-\alpha} L(x) \triangleq K_{4,1}  (\rho \vee \rho_\alpha)^{ n} x^{-\alpha} L(x) . 
\end{equation*}
for all $x \geq \max\{ x_0, a_0 x_0\}$. For the values $1 \leq x \leq \max\{x_0, a_0 x_0\}$ use Markov's inequality to obtain
$$P(Z_N > x/a_n) \leq \frac{a_n^{\alpha-\delta}}{x^{\alpha-\delta}} \leq K_{4,1} (\rho \vee \rho_\alpha)^n x^{-\alpha+\delta} \leq K_{4,1} \sup_{1 \leq t \leq \max\{x_0, a_0 x_0\}} \frac{t^{\delta}}{L(t)} \cdot x^{-\alpha} L(x).$$
Setting $K_4 = K_{4,1} \max\left\{1, \, \sup_{1 \leq t \leq \max\{x_0, a_0 x_0\}} t^\delta L(t)^{-1} \right\}$ gives the statement of the lemma. 
\end{proof}

\bigskip

We are now ready to give the proof of Lemma \ref{L.Bound1Iter}. 

\begin{proof}[Proof of Lemma \ref{L.Bound1Iter}] 
First recall that $W_{n+1} \stackrel{\mathcal{D}}{=} \sum_{i=1}^N C_i W_{n,i}$, where the $W_{n,i}$ are i.i.d. having the same distribution as $W_n$ and are independent of the vector $(N, C_1, \dots, C_N)$. The idea of the proof is to split  $\{ \sum_{i=1}^N C_i W_{n, i}^+ > x\}$ into several different events, and bound each of them separately. We proceed as follows,
\begin{align}
P\left( W_{n+1}^+ > x \right) &\leq P\left( \sum_{i=1}^N C_i W_{n,i}^+ > x \right) \notag \\
&\leq P\left( \sum_{i=1}^N C_i W_{n,i}^+ > x, \, Z_N \leq x/a_n \right) + P\left( Z_N > x/a_n \right) \notag \\
&\leq P\left( \sum_{i=1}^N C_i W_{n,i}^+ > x, \, Z_N \leq x/a_n, \, I_N(w/\gamma_n) = 0 \right) \notag \\
&\hspace{5mm} + P\left( I_N(w/\gamma_n) \geq 1 \right)  + P\left( Z_N > x/a_n \right) \notag \\
&\leq P\left(\sum_{i=1}^N C_i W_{n,i}^+ > x, \, J_N(y) = 0, \, Z_N \leq x/a_n, \, I_N(w/\gamma_n) = 0 \right) \label{eq:FirstIneq} \\
&\hspace{5mm} + P\left(\sum_{i=1}^N C_i W_{n,i}^+ > x, \, J_N(y) = 1, \, Z_N \leq x/a_n, \, I_N(w/\gamma_n) = 0 \right) \label{eq:SecondIneq} \\
&\hspace{5mm} + P\left( J_N(y) \geq 2, \, Z_N \leq x/a_n, \, I_N(w/\gamma_n) = 0 \right) \notag  \\
&\hspace{5mm}  + P\left( I_N(w/\gamma_n) \geq 1 \right)  + P\left( Z_N > x/a_n \right) . \notag  
\end{align}

Note that the probability in \eqref{eq:SecondIneq} is bounded by
\begin{align*}
&P\left(\sum_{i=1}^N C_i W_{n,i}^+ > x, \, J_N(y) = 1, \, J_N((1-\delta)x) = 0, \, Z_N \leq x/a_n, \, I_N(w/\gamma_n) = 0 \right)  \\
&\hspace{5mm} + P\left(J_N((1-\delta)x) \geq 1,  \, I_N(w/\gamma_n) = 0 \right)  \\
&\leq P\left(\sum_{i=1}^N C_i W_{n,i}^+ 1(C_i W_{n,i}^+ \leq y)  > \delta x, \, Z_N \leq x/a_n, \, I_N(w/\gamma_n) = 0 \right)  \\
&\hspace{5mm} + P\left(J_N((1-\delta)x) \geq 1, \, I_N(x/T) = 0 \right) + P\left( I_N(w/\gamma_n) = 0, \, I_N(x/T) \geq 1 \right),
\end{align*}
while \eqref{eq:FirstIneq} is bounded by
\begin{align*}
P\left(\sum_{i=1}^N C_i W_{n,i}^+ 1(C_i W_{n,i}^+ \leq y) > x,  \, Z_N \leq x/a_n, \, I_N(w/\gamma_n) = 0 \right).
\end{align*}
It follows that
\begin{align}
P(W_{n+1}^+ > x) &\leq 2 P\left(\sum_{i=1}^N C_i W_{n,i}^+ 1(C_i W_{n,i}^+ \leq y)  > \delta x, \, Z_N \leq x/a_n, \, I_N(w/\gamma_n) = 0 \right) \label{eq:TruncSum} \\
&\hspace{5mm} + P\left( J_N(y) \geq 2, \, Z_N \leq x/a_n, \, I_N(w/\gamma_n) = 0 \right) \label{eq:SecondMax}  \\
&\hspace{5mm} + P\left(I_N(w/\gamma_n) \geq 1 \right) + P\left( I_N(w/\gamma_n) = 0, \, I_N(x/T) \geq 1 \right) \label{eq:Weights} \\
&\hspace{5mm} +  P\left( Z_N > x/a_n \right) \label{eq:Potter} \\
&\hspace{5mm} +  P\left(J_N((1-\delta)x) \geq 1, \, I_N(x/T) = 0 \right)  . \notag
\end{align}
By Lemmas \ref{L.TruncSum}, \ref{L.SecondMax}, \ref{L.Weights} and \ref{L.Potter}, \eqref{eq:TruncSum} + \eqref{eq:SecondMax} + \eqref{eq:Weights} + \eqref{eq:Potter} is bounded by
\begin{align*}
(2 K_1 + K_2 + K_3 + K_4) (\rho \vee \rho_\alpha)^n x^{-\alpha} L(x) + K_3 \, x^{-\alpha-\epsilon/2} L(x)
\end{align*}
for all $x \geq 1$ and all $0 \leq n \leq  \frac{\epsilon}{2|\log(\rho \vee \rho_\alpha)|} \log x$, where $K_1, K_2, K_3, K_4$ are finite constants that only depend on $\epsilon$, $\delta$ and $T$. Moreover, for this range of values of $n$ we have 
$$x^{-\epsilon/2} = (\rho \vee \rho_\alpha)^{ \frac{\epsilon}{2|\log(\rho \vee \rho_\alpha)|}  \log x} \leq (\rho \vee \rho_\alpha)^n.$$
Define $K_0 = K_0(\delta, \epsilon) = 2K_1 + K_2 + 2 K_3 + K_4$ to obtain that
\begin{align}
P(W_{n+1}^+ > x) &\leq K_0 (\rho \vee \rho_\alpha)^n x^{-\alpha} L(x) \notag \\
&\hspace{5mm} +  P\left(J_N((1-\delta)x) \geq 1, \, I_N(x/T) = 0 \right). \label{eq:OneSummand} 
\end{align}

To bound \eqref{eq:OneSummand} use the union bound to obtain 
\begin{align}
P\left(J_N((1-\delta)x) \geq 1, \, I_N(x/T) = 0 \right) &= E\left[ 1(I_N(x/T) = 0) \cdot 1\left( \bigcup_{i=1}^N \{ C_i W_{n,i}^+ > (1-\delta)x \} \right) \right]  \notag \\
&\leq E\left[ 1(I_N(x/T) = 0) \sum_{i=1}^N 1\left( C_i W_{n,i}^+ > (1-\delta)x  \right)  \right] ,   \label{eq:Bound2}
\end{align}
which completes the proof. 
\end{proof}

\bigskip

We can now give the proof of Proposition \ref{P.UniformBound}, the main technical contribution of the paper. 

\begin{proof}[Proof of Proposition \ref{P.UniformBound}]
Recall that $\overline{G}(x) = P(Z_N > x)$. Note that it is enough to prove the proposition for all $x \geq x_1$ for some $x_1 = x_1(\eta,\epsilon) \geq 1$, since for all $1 \leq x \leq x_1$ and $n \geq 1$, 
\begin{align*}
P(W_n^+ > x) &= \frac{P(W_n^+ > x)}{ \eta^n \overline{G}(x)} \, \eta^n \overline{G}(x)\\
&\leq \frac{E[Q^+] \rho^n x^{-1}}{\eta^n \overline{G}(x)} \, \eta^n \overline{G}(x) \qquad \text{(by Markov's inequality)} \\
&\leq \sup_{1 \leq t \leq x_1} \frac{E[Q^+]}{t \overline{G}(t)} \, \cdot \eta^n \overline{G}(x) .
\end{align*}

Next, choose $0 < \delta < \min\{ (\alpha-1)/2, \epsilon, 1/2\}$ such that
\begin{equation} \label{eq:EpsilonChoice}
\rho_\alpha \left( \delta + (1-\delta)^{-\alpha-1} \right) \leq \eta.
\end{equation}

Now note that by Potter's Theorem (see Theorem 1.5.6 (iii) in \cite{BiGoTe1987}), there exists a constant $x_0 = x_0(2,\delta) > 0$ such that
\begin{align*}
E\left[  \sum_{i=1}^N  \frac{\overline{G}((1-\delta) x/ C_i)}{\overline{G}(x)}   \right] &\leq E\left[ \sum_{i=1}^N 2 (1-\delta)^{-\alpha} C_i^\alpha \max\{ ((1-\delta)/C_i)^{-\delta}, ((1-\delta)/C_i)^\delta \} \right] \\ &\leq 2(1-\delta)^{-\alpha-\delta} (\rho_{\alpha-\delta} + \rho_{\alpha+\delta}) < \infty
\end{align*}
for all $x \geq x_0$. And for $1 \leq x \leq x_0$ Markov's inequality gives
$$E\left[  \sum_{i=1}^N  \frac{\overline{G}((1-\delta) x/ C_i)}{\overline{G}(x)}   \right] \leq \frac{1}{\overline{G}(x)} E\left[ \sum_{i=1}^N \frac{E[Z_N^{\alpha-\delta}] C_i^{\alpha-\delta}}{(1-\delta)^{\alpha-\delta} x^{\alpha-\delta}} \right] \leq  \frac{ E[Z_N^{\alpha-\delta}] \rho_{\alpha-\delta} }{(1-\delta)^{\alpha-\delta}} \sup_{1 \leq t \leq x_0} \frac{t^{-\alpha+\delta}}{\overline{G}(t)} < \infty.$$
Hence, by dominated convergence, 
$$\lim_{x \to \infty} E\left[  \sum_{i=1}^N  \frac{\overline{G}((1-\delta) x/ C_i)}{\overline{G}(x)}   \right]  = E\left[  \sum_{i=1}^N  \lim_{x \to \infty} \frac{\overline{G}((1-\delta) x/ C_i)}{\overline{G}(x)}   \right] = (1-\delta)^{-\alpha} \rho_\alpha.$$
It follows that there exists $x_1 = x_1(\delta) \geq 1$ for which 
\begin{equation} \label{eq:Limit}
E\left[  \sum_{i=1}^N  \frac{\overline{G}((1-\delta) x/ C_i)}{\overline{G}(x)}   \right] \leq (1-\delta)^{-\alpha-1} \rho_\alpha
\end{equation}
for all $x \geq x_1$. Set $T = 2 x_1$.  

Now, by Lemma \ref{L.Bound1Iter}, there exists a finite constant $K_0 > 0$ (that does not depend on $n$) such that
$$P(W_{n+1}^+ > x) \leq K_0 (\rho \vee \rho_\alpha)^n \overline{G}(x) + E\left[ 1(I_N(x/T) = 0) \sum_{i=1}^N 1\left( C_i W_{n,i}^+ > (1-\delta)x  \right)  \right]$$
for all $x \geq 1$ and $0 \leq n  \leq \frac{\epsilon}{2|\log (\rho \vee \rho_\alpha)|} \log x$.  Let $K_1 =  (\delta \rho_\alpha)^{-1} K_0$ to obtain 
\begin{equation} \label{eq:oneIter}
P(W_{n+1}^+ > x) \leq K_1 \delta \rho_\alpha \eta^{n} \overline{G}(x) + E\left[ 1(I_N(x/T) = 0) \sum_{i=1}^N 1\left( C_i W_{n,i}^+ > (1-\delta)x  \right)  \right]
\end{equation}
for all $x \geq 1$ and $0 \leq n  \leq \frac{\epsilon}{2|\log (\rho \vee \rho_\alpha)|} \log x$. 

Now we go on to derive bounds for $P(W_n^+ > x)$ for different ranges of $n$. For the values $1 \leq n \leq  \frac{\epsilon}{2|\log(\rho \vee \rho_\alpha)|} \log x$ we proceed by induction. Let $\mathcal{F} = \sigma(N, C_1, \dots, C_N)$. Define
$$K_2 = \max\left\{ K_1,   \, K_1 \delta  + E[ (Q^+)^{\alpha+\epsilon}] \frac{\rho_{\alpha+\epsilon}}{\eta} \sup_{t \geq 1} \frac{1}{t^\epsilon L(t)}  \right\}.$$

For $n = 1$, we have by \eqref{eq:oneIter},
\begin{align*}
P(W_1^+ > x) &\leq K_1 \delta \rho_\alpha \overline{G}(x) + E\left[ 1(I_N(x/T) = 0) \sum_{i=1}^N 1\left( C_i W_{0,i}^+ > (1-\delta)x  \right)  \right],
\end{align*}
where $W_{0,i}^+ = Q_i^+$ and $\{Q_i^+\}$ are independent of $(N, C_1, \dots, C_N)$. By conditioning on $\mathcal{F}$ we get
\begin{align*}
E\left[ 1(I_N(x/T) = 0) \sum_{i=1}^N 1\left( C_i W_{0,i}^+ > (1-\delta)x  \right)  \right] &= E\left[ 1(I_N(x/T) = 0) \sum_{i=1}^N E[ 1\left( C_i Q_{i}^+ > (1-\delta)x  \right) | \mathcal{F}] \right] \\
&\leq E\left[ \sum_{i=1}^N \frac{E[ (C_i Q_i^+)^{\alpha+\epsilon} | C_i]}{x^{\alpha+\epsilon}} \right] \qquad \text{(by Markov's inequality)} \\
&= E[ (Q^+)^{\alpha+\epsilon}] \rho_{\alpha+\epsilon} x^{-\alpha-\epsilon} \\
&\leq E[ (Q^+)^{\alpha+\epsilon}] \rho_{\alpha+\epsilon} \sup_{t \geq 1} \frac{1}{t^\epsilon L(t)} \overline{G}(x).
\end{align*}
It follows that 
$$P(W_1^+ > x) \leq \left( K_1 \delta  + E[ (Q^+)^{\alpha+\epsilon}] \frac{\rho_{\alpha+\epsilon}}{\eta} \sup_{t \geq 1} \frac{1}{t^\epsilon L(t)} \right) \eta \overline{G}(x) \leq K_2 \eta \overline{G}(x)$$
for all $x \geq 1$. Suppose now that
\begin{equation} \label{eq:InductionHyp}
\overline{G}_n(x) \triangleq P(W_n^+ > x) \leq K_2 \eta^n \overline{G}(x) 
\end{equation}
for all $x \geq x_1$. 

Let $2 \leq n \leq  \frac{\epsilon}{2|\log(\rho \vee \rho_\alpha)|} \log x$. Then, by the induction hypothesis \eqref{eq:InductionHyp}, we have for all $x \geq x_1$,
\begin{align*}
E\left[ 1(I_N(x/T) = 0) \sum_{i=1}^N 1\left( C_i W_{n,i}^+ > (1-\delta)x  \right)  \right]  &= E\left[  1(I_N(x/T) = 0) \sum_{i=1}^N  E\left[ \left. 1\left( C_i W_{n,i}^+ > (1-\delta)x  \right) \right| \mathcal{F} \right]  \right] \\
&= E\left[  1(I_N(x/T) = 0) \sum_{i=1}^N  \overline{G}_n((1-\delta) x/ C_i)  \right] \\
&\leq K_2 \eta^n E\left[  1(I_N(x/T) = 0) \sum_{i=1}^N  \overline{G}((1-\delta) x/ C_i)  \right] \\
&= K_2 \eta^n \overline{G}(x) E\left[ 1(I_N(x/T) = 0) \sum_{i=1}^N  \frac{\overline{G}((1-\delta) x/ C_i)}{\overline{G}(x)}   \right]  \\
&\leq K_2 \eta^n (1-\delta)^{-\alpha-1} \rho_\alpha \overline{G}(x) ,
\end{align*}
where in the last inequality we used \eqref{eq:Limit}. Then, by \eqref{eq:oneIter}
\begin{align*}
P(W_{n+1}^+ > x) &\leq K_1 \delta \rho_\alpha \eta^n \overline{G}(x) + K_2 \eta^n (1-\delta)^{-\alpha-1} \rho_\alpha \overline{G}(x) \\
&\leq K_2 \left( \delta + (1-\delta)^{-\alpha-1} \right) \rho_\alpha \eta^n \overline{G}(x) \\
&\leq K_2 \eta^{n+1} \overline{G}(x),
\end{align*}
for all $x \geq x_1$. 

Finally, for $n \geq  \frac{\epsilon}{2|\log(\rho \vee \rho_\alpha)|} \log x$, we use the moment estimates for $W_n$. Define
$$\varepsilon =  \frac{\eta}{\rho \vee \rho_\alpha}  - 1 > 0 \qquad \text{and} \qquad \kappa = \frac{\epsilon  \log(1+\varepsilon)}{2|\log(\rho \vee \rho_\alpha)|}.$$
Choose $0 < s < \min\{\kappa/2, \alpha-1\}$. Then, by Markov's inequality and Proposition \ref{P.GeneralMoments}, we have
\begin{align}
P(W_n^+ > x) &\leq E[(W_n^+)^{\alpha-s}] x^{-\alpha+s} \notag \\
&\leq K_{\alpha-s} (\rho \vee \rho_{\alpha-s})^n x^{-\alpha+s} \notag \\
&\leq K_{\alpha-s} (1+\varepsilon)^{- n} \eta^n x^{-\alpha+s} \notag \\
&\leq K_{\alpha-s} x^{- \log(1+\varepsilon) \frac{\epsilon}{2|\log(\rho \vee \rho_\alpha)|}}  \eta^n x^{-\alpha+s} \notag  \\
&= K_{\alpha-s}   \eta^n x^{-\alpha-\kappa+s} \label{eq:boundForW_n}
\end{align}
for all $x > 0$. Our choice of $s$ now gives
\begin{align*}
P(W_n^+ > x) &\leq K_{\alpha-s}   \eta^n x^{-\alpha-\kappa/2} \leq K_{\alpha-s} \sup_{t \geq 1} \frac{t^{-\kappa/2}}{L(t)} \cdot \eta^n \overline{G}(x) \triangleq K_3 \eta^n \overline{G}(x)
\end{align*}
for all $x \geq 1$. 

We have thus shown that
$$P(W_n^+ > x) \leq \max\{ K_2, K_3\} \eta^n \overline{G}(x)$$
for all $x \geq x_1$ and $n \geq 1$. 
\end{proof}

\bigskip

We end this section with a sketch of the proof of Proposition \ref{P.UniformBoundQ}. As mentioned before, the proofs of the other results presented in Section \ref{S.QDominates} have been omitted since they are very similar to those from Section \ref{S.NDominates}.

\begin{proof}[Sketch of the proof of Proposition \ref{P.UniformBoundQ}]
Follow the proof of Proposition \ref{P.UniformBound} up to inequality \eqref{eq:Limit} substituting $\overline{G}(x) = P(Z_N > x)$ with $\overline{F}(x) \triangleq P(Q > x) = x^{-\alpha} L(x)$. Now note that by Markov's inequality 
$$P(Z_N > x) \leq E[Z_N^{\alpha+\epsilon}] x^{-\alpha-\epsilon}$$
for all $x > 0$, so we can use Lemma \ref{L.Bound1Iter} to obtain 
\begin{equation*} 
P(W_{n+1}^+ > x) \leq K_0 (\rho \vee \rho_\alpha)^n E[Z_N^{\alpha+\epsilon}] x^{-\alpha-\epsilon} + E\left[ 1(I_N(x/T) =0) \sum_{i=1}^N 1(C_i W_{n,i}^+ > (1-\delta)x ) \right]
\end{equation*}
for all $x \geq 1$ and all $0 \leq n \leq \frac{\epsilon}{2|\log(\rho \vee \rho_\alpha)|} \log x$; $K_0 >0$ is a constant that does not depend on $n$. Let $K_1 = (\delta \rho_\alpha)^{-1} K_0 E[Z_N^{\alpha+\epsilon}] \sup_{t \geq 1} t^{-\epsilon}/L(t)$ to derive
\begin{equation} \label{eq:Q_1Iter}
P(W_{n+1}^+ > x) \leq K_1 \delta \rho_\alpha \eta^n \overline{F}(x) + E\left[ 1(I_N(x/T) =0) \sum_{i=1}^N 1(C_i W_{n,i}^+ > (1-\delta)x ) \right]
\end{equation}
for all $x \geq 1$ and all $0 \leq n \leq \frac{\epsilon}{2|\log(\rho \vee \rho_\alpha)|} \log x$.

Now define $\mathcal{F} = \sigma(N, C_1, \dots, C_N)$ and $K_2 = \max\{ K_1, 1\}$.  For the values $0 \leq n \leq \frac{\epsilon}{2|\log(\rho \vee \rho_\alpha)|} \log x$ we proceed by induction. For $n=1$ we have $W_{0,i}^+ = Q_i^+$, with the $\{Q_i^+\}$ independent of $(N, C_1, \dots, C_N)$. By conditioning on $\mathcal{F}$ and using \eqref{eq:Limit} (with $\overline{G}(x)$ substituted by $\overline{F}(x)$), we get
\begin{align*}
E\left[ 1(I_N(x/T) =0) \sum_{i=1}^N 1(C_i W_{0,i}^+ > (1-\delta)x ) \right] &= E\left[ 1(I_N(x/T) =0) \sum_{i=1}^N \overline{F}((1-\delta)x/C_i )  \right] \\
&\leq (1-\delta)^{-\alpha-1} \rho_\alpha \overline{F}(x). 
\end{align*}
It follows that
$$P(W_1^+ > x ) \leq K_1 \delta \rho_\alpha \overline{F}(x) + (1-\delta)^{-\alpha-1} \rho_\alpha \overline{F}(x) \leq K_2 \eta \overline{F}(x)$$
for all $x \geq 1$. 

The rest of the proof continues exactly as that of Proposition \ref{P.UniformBound} with $\overline{G}(x)$ substituted by $\overline{F}(x)$. 
\end{proof}

\appendix
\section{Some results for weighted random sums}

We include in this appendix two results related to the asymptotic behavior of randomly weighted and randomly stopped sums. The first one is a quick corollary of a theorem from \cite{Olvera_11} that allows the addition of the $Q$ term for the case where $Z_N$ has a regularly varying distribution. The second one also uses some of the results from \cite{Olvera_11}, but is more involved since it refers to the case where $Q$ has a regularly varying distribution. Both of these results may be of independent interest.

\begin{thm} \label{T.SumPlusQ_ZN}
Let $\{X_i\}$ be a sequence of i.i.d. random variables with common distribution $\overline{F} \in \mathcal{R}_{-\alpha}$,  $\alpha > 1$, $E[(X_1^-)^{1+\epsilon}]<\infty$ for some $0 < \epsilon < \alpha-1$, and $E[X_1] > 0$.  Assume further that $(Q, N, C_1, \dots, C_N)$ is a random vector, independent of the $\{X_i\}$, with $N \in \mathbb{N} \cup \{\infty\}$, $\{C_i\} \geq 0$, and $Q \in \mathbb{R}$. Then, if $Z_N = \sum_{i=1}^N C_i$ satisfies $P(Z_N > x) \sim c P(X_1 > x)$ for some $c > 0$, $E\left[ \sum_{i=1}^N C_i^{\alpha+\epsilon} \right] < \infty$ and $E[|Q|^{\alpha+\epsilon}] < \infty$, we have
$$P\left( \sum_{i=1}^N C_i X_i + Q > x \right) \sim \left( E\left[ \sum_{i=1}^N C_i^\alpha \right]  +  c (E[X_1])^{\alpha} \right) \overline{F}(x)$$
as $x \to \infty$. 
\end{thm}

\begin{proof}
Let $S_N = \sum_{i=1}^N C_i X_i$, and note that since $\alpha-\epsilon > 1$, the inequality $\sum_{i=1}^k y_i^\beta \leq \left( \sum_{i=1}^k y_i \right)^\beta$ for $y_i \geq 0$ and any $\beta \geq 1$ gives $E\left[ \sum_{i=1}^N C_i^{\alpha-\epsilon} \right] \leq E\left[ Z_N^{\alpha-\epsilon}\right]$, which is finite by the assumption $P(Z_N > x) \sim c \overline{F}(x)$. Then, by Theorem 2.5 and the remark after it in \cite{Olvera_11} ,
\begin{align*}
P(S_N + Q > x) &\leq P(S_N + Q > x, Q \leq x/\log x) + P(Q > x/\log x) \\
&\leq P(S_N > x-x/\log x) + \frac{E[|Q|^{\alpha+\epsilon}]}{(x/\log x)^{\alpha+\epsilon}} \qquad \text{(by Markov's inequality)} \\
&\sim E\left[ \sum_{i=1}^N C_i^{\alpha} \right] \overline{F}(x-x/\log x) + P(Z_N > (x-x/\log x)/E[X_1]) + o\left( \overline{F}(x) \right) \\
&\sim  E\left[ \sum_{i=1}^N C_i^{\alpha} \right] \overline{F}(x) + c (E[X_1])^\alpha \overline{F}(x).
\end{align*}
For the lower bound, the same arguments give
\begin{align*}
P(S_N + Q > x) &\geq P(S_N + Q > x, \, Q \geq -x/\log x) \\
&\geq P(S_N > x + x/\log x) - P( Q < -x/\log x ) \\
&\geq P(S_N > x + x/\log x) - \frac{E[|Q|^{\alpha+\epsilon}]}{(x/\log x)^{\alpha+\epsilon}} \\
&\sim E\left[ \sum_{i=1}^N C_i^{\alpha} \right] \overline{F}(x) + c (E[X_1])^\alpha \overline{F}(x).
\end{align*} 
\end{proof}

\bigskip

\begin{thm} \label{T.SumPlusQ_Q}
Let $\{X_i\}$ be a sequence of i.i.d. random variables with common distribution $\overline{F} \in \mathcal{R}_{-\alpha}$,  $\alpha > 1$, $E[(X_1^-)^{1+\epsilon}]<\infty$ for some $\epsilon > 0$.  Assume further that $(Q, N, C_1, \dots, C_N)$ is a random vector, independent of the $\{X_i\}$, with $N \in \mathbb{N} \cup \{\infty\}$, $\{C_i\} \geq 0$, and $Q \in \mathbb{R}$. Then, if $P(Q > x) \sim c P(X_1 > x)$ for some $c > 0$, and $Z_N = \sum_{i=1}^N C_i$ satisfies $E\left[ Z_N^{\alpha+\epsilon} \right] < \infty$, we have
$$P\left( \sum_{i=1}^N C_i X_i + Q > x \right) \sim \left( E\left[ \sum_{i=1}^N C_i^\alpha  \right] + c \right)\overline{F}(x)$$
as $x \to \infty$. 
\end{thm}

\begin{proof}
Let $S_N = \sum_{i=1}^N C_i X_i$ and define $J_N(t) = \#\{ 1 \leq i < N+1: C_i X_i > t\}$. Assume that $0 < \epsilon < \alpha-1$ and set $\nu = \epsilon/(2(\alpha+\epsilon))$, $\gamma = \left( E[|X_1|^{1+\epsilon}] \right)^{1/(1+\epsilon)}$, $w = x^{1-\nu}/\gamma$, $y = x/\log x$ and $\delta = 1/\sqrt{\log x}$. Also note that
$$E\left[ \sum_{i=1}^N C_i^{\alpha-\epsilon} \right] \leq E\left[ Z_N^{\alpha-\epsilon}\right] \leq \left( E\left[ Z_N^{\alpha+\epsilon} \right] \right)^\frac{\alpha+\epsilon}{\alpha-\epsilon} < \infty.$$
Then, 
\begin{align}
P\left( S_N + Q > x \right) &\leq P\left(S_N + Q > x, \, S_N > (1-\delta)x  \right) + P\left(S_N + Q > x, \, S_N \leq (1-\delta)x, \, Q > (1-\delta)x  \right) \notag \\
&\hspace{5mm} + P\left(S_N + Q > x, \, S_N \leq (1-\delta)x, \, Q \leq (1-\delta)x  \right) \notag \\
&\leq P(S_N > (1-\delta)x) + P(Q > (1-\delta)x) \label{eq:MainTermsUB} \\
&\hspace{5mm} + P\left(S_N + Q > x, \, S_N \leq (1-\delta)x, \, \delta x < Q \leq (1-\delta)x   \right). \label{eq:MixTermsUB}
\end{align}
By Theorem 2.3 and the remark following Theorem 2.5 in \cite{Olvera_11}, we have that \eqref{eq:MainTermsUB} is equal to
$$E\left[\sum_{i=1}^N C_i^\alpha \right] \overline{F}(x) + c \overline{F}(x) + o\left( \overline{F}(x) \right)$$
as $x \to \infty$. To analyze \eqref{eq:MixTermsUB} first note that it is bounded by
\begin{align}
P\left( S_N > \delta x, \, Q > \delta x \right) &\leq P\left( S_N > \delta x, \, Q > \delta x, \, Z_N \leq w \right)  + P(Z_N > w)  \notag \\
&\leq P\left( S_N > \delta x, \, Q > \delta x, \, Z_N \leq w, \, J_N(y) = 0 \right) \notag \\
&\hspace{5mm} + P\left( Q > \delta x, \, Z_N \leq w,  \, J_N(y) \geq 1 \right) + P(Z_N > w) \notag \\
&\leq P\left( \sum_{i=1}^N C_i X_i^+ > \delta x, \, J_N(y) = 0, \, Z_N \leq w \right) \notag \\
&\hspace{5mm} + P\left( Q > \delta x, \, Z_N \leq w, \, J_N(y) \geq 1 \right) + P(Z_N > w) \notag \\
&\leq P\left( \sum_{i=1}^N C_i X_i^+ 1(C_i X_i^+ \leq y) > \delta x, \, Z_N \leq w \right) \label{eq:HeavyTrunc} \\
&\hspace{5mm} + P\left( Q > \delta x, \, Z_N \leq w, \, J_N(y) \geq 1 \right)  \label{eq:QandJ} \\
&\hspace{5mm} + P(Z_N > w)  .\notag
\end{align}
Now, by Lemma 3.4 in \cite{Olvera_11} (note that $Z_N \leq w$ implies $I_N(w) = \#\{1 \leq i < N+1: C_i > w \} = 0$), \eqref{eq:HeavyTrunc} is bounded by $K x^{-h}$ for any $h > 0$, in particular, for $h = \alpha+\epsilon$, from where it follows that it is $o\left( \overline{F}(x) \right)$. Here and in the remainder of the proof $K > 0$ is a generic constant, not necessarily the same from one line to the next. To analyze \eqref{eq:QandJ} let $\mathcal{F} = \sigma(Q, N, C_1, \dots, C_N)$ and note that we can write the probability as
\begin{align*}
&E\left[ 1(Q > \delta x, \, Z_N \leq w) E\left[ 1(J_N(y) \geq 1) | \mathcal{F} \right] \right] \\
&\leq E\left[ 1(Q > \delta x, \, Z_N \leq w) \sum_{i=1}^N E\left[ 1(C_i X_i > y) | \mathcal{F} \right] \right] \qquad \text{(by the union bound)} \\
&\leq \frac{E[|X_1|^{1+\epsilon}]}{y^{1+\epsilon}} E\left[ 1(Q > \delta x, \, Z_N \leq w) \sum_{i=1}^N C_i^{1+\epsilon} \right] \qquad \text{(by Markov's inequality)} \\
&\leq \frac{K }{y^{1+\epsilon}} E\left[ 1(Q > \delta x, Z_N \leq w) Z_N^{1+\epsilon} \right] \leq \frac{K w^{1+\epsilon}}{y^{1+\epsilon}} P(Q > \delta x) \\
&\leq \frac{K (\log x)^{1+\epsilon}}{x^{(1+\epsilon)\epsilon\nu}} \overline{F}(\delta x) \leq \frac{K (\log x)^{1+\epsilon} }{x^{(1+\epsilon)\epsilon\nu}\delta^{\alpha+\epsilon}} \overline{F}(x) \\
&= \frac{K(\log x)^{1+\alpha/2+3\epsilon/2}}{x^{(1+\epsilon)\epsilon\nu}} \overline{F}(x) = o \left( \overline{F}(x) \right), 
\end{align*}
where in the sixth inequality we used Potter's Theorem (see Theorem 1.5.6 in \cite{BiGoTe1987}). Finally, from Markov's inequality we get
\begin{align*}
P(Z_N > w)  &\leq \frac{E[Z_N^{\alpha+\epsilon}]}{w^{\alpha+\epsilon}} \leq \frac{K}{x^{(1-\nu)(\alpha+\epsilon)}} = \frac{K}{x^{\alpha+\epsilon/2}} = o\left( \overline{F}(x) \right).
\end{align*}
We have thus shown that \eqref{eq:MixTermsUB} is $o\left( \overline{F}(x) \right)$, and the upper bound follows. 

For the lower bound we have that
\begin{align}
P\left(S_N + Q > x \right) &\geq P\left(S_N + Q > x, \, Z_N \leq w, \, S_N > (1+\delta)x \right) \notag \\
&\hspace{5mm} + P\left(S_N + Q > x, \ Z_N \leq w, \, S_N \leq (1+\delta)x, \, Q > (1+\delta)x \right) \notag\\
&= P\left(S_N > (1+\delta)x, \, Z_N \leq w \right) + P\left(Q > (1+\delta)x, \, Z_N \leq w \right)   \label{eq:mainTerms} \\
&\hspace{5mm} - P\left(S_N + Q \leq x, \, Z_N \leq w, \, S_N > (1+\delta)x \right) \label{eq:harderTerm} \\
&\hspace{5mm} - P\left(S_N  + Q \leq x, \, Z_N \leq w, \, S_N \leq (1+\delta)x, \, Q > (1+\delta)x \right) \label{eq:Markov1} \\
&\hspace{5mm} - P\left(Z_N \leq w, \, S_N > (1+\delta)x, \, Q > (1+\delta)x \right) \label{eq:Markov2}
\end{align}
Note that \eqref{eq:mainTerms} is bounded from below by
$$P(S_N > (1+\delta)x) + P(Q > (1+\delta)x ) - 2 P(Z_N > w) = E\left[ \sum_{i=1}^N C_i^\alpha \right] \overline{F}(x) + c \overline{F}(x) + o\left( \overline{F}(x) \right),$$
by the same arguments used for the upper bound. Also note that we can bound the sum of the probabilities in \eqref{eq:Markov1} and \eqref{eq:Markov2} by
\begin{align*}
&P(S_N \leq -\delta x, \, Z_N \leq w, \, Q > x) + P(S_N > \delta x, \, Z_N \leq w, \, Q > x) \\
&\leq 2P\left( Z_N \leq w, \, |S_N| \geq \delta x, \, Q > x \right) \\
&= 2 E\left[ 1(Z_N \leq w, \, Q > x) E\left[ 1(|S_N| \geq \delta x) | \mathcal{F} \right] \right] \\
&\leq \frac{2}{\delta x} E\left[ 1(Z_N \leq w, \, Q > x) E\left[ |S_N| | \mathcal{F} \right] \right] \qquad \text{(by Markov's inequality)} \\
&\leq \frac{2E[|X_1|]}{\delta x} E\left[ 1(Z_N \leq w, \, Q > x) Z_N \right] \\
&\leq \frac{Kw}{\delta x} P(Q > x) \leq \frac{K(\log x)^{1/2}}{x^\nu} \overline{F}(x) = o\left( \overline{F}(x) \right).
\end{align*}
It only remains to analyze \eqref{eq:harderTerm}. Let $\kappa = \nu^2$ and note that  the probability in \eqref{eq:harderTerm} is bounded by
\begin{align}
&P\left(S_N+Q \leq  x, \, Z_N \leq w, \, S_N > (1+\delta)x, \, J_N(\kappa x) = 0 \right) \notag \\
&\hspace{5mm} + P\left(S_N+Q \leq x, \, Z_N \leq w, \, S_N > (1+\delta)x, \, J_N(\kappa x) \geq 1 \right) \notag  \\
&\leq  P\left( \sum_{i=1}^N C_i X_i 1(C_i X_i \leq \kappa x) > (1+\delta) x, \, Z_N \leq w \right) \label{eq:VeryTrunc} \\
&\hspace{5mm} + P\left(Q < -\delta x, \, Z_N \leq w, \, S_N > (1+\delta)x, \, J_N(\kappa x) \geq 1 \right) \label{eq:QandJ_2} 
\end{align}
By Lemma 3.2 in \cite{Olvera_11}, with $u = x^{1-\nu}$, $v = \kappa x$, $z = x$, $\eta = 1+\epsilon$ and $A = (-\infty, w]$ (note that $Z_N \leq w$ implies $I_N(w) = 0$), \eqref{eq:VeryTrunc} is bounded by
\begin{align*}
E\left[ 1(Z_N \leq w) e^{-\frac{\epsilon}{\kappa x} \log(\kappa x^\nu) \left( x - \left(E[X_1] +  \frac{K\gamma}{\log(\kappa x^\nu)} \right)^+ Z_N \right)} \right] &\leq K e^{-\frac{\epsilon}{\kappa} \log(\kappa x^\nu)} \leq \frac{K}{x^{2(\alpha+\epsilon)}} = o\left( \overline{F}(x) \right).
\end{align*}

As for \eqref{eq:QandJ_2} use Potter's Theorem (see Theorem 1.5.6(iii) in \cite{BiGoTe1987}) to obtain,
\begin{align*}
&P(Q < -\delta x, \, Z_N \leq w, \, J_N(\kappa x) \geq 1 ) \\
&\leq E\left[ 1(Q < -\delta x, \, Z_N \leq w) \sum_{i=1}^N \overline{F}(\kappa x/C_i) \right] \qquad \text{(by the union bound)} \\
&\leq K\overline{F}(x) E\left[ 1(Q < -\delta x) \sum_{i=1}^N C_i^{\alpha+\epsilon} \right] \qquad \text{(by Potter's Theorem)} \\
&\leq K\overline{F}(x) E\left[ 1(Q < -\delta x) Z_N^{\alpha+\epsilon} \right] \\
&= o\left( \overline{F}(x) \right), 
\end{align*}
where in the last step we used dominated convergence ($E[Z_N^{\alpha+\epsilon}] < \infty$) to see that $E\left[ 1(Q < -\delta x) Z_N^{\alpha+\epsilon} \right] \to 0$ as $x \to \infty$. 
\end{proof}



\bibliographystyle{plain}

\begin{thebibliography}{10}

\bibitem{Aldo_Band_05}
D.J. Aldous and A.~Bandyopadhyay.
\newblock A survey of max-type recursive distributional equation.
\newblock {\em Annals of Applied Probability}, 15(2):1047--1110, 2005.

\bibitem{Alsm_Bigg_Mein_10}
G.~Alsmeyer, J.D. Biggins, and M.~Meiners.
\newblock The functional equation of the smoothing transform.
\newblock {\em arXiv:0906.3133}, 2010.

\bibitem{Alsm_Mein_10b}
G.~Alsmeyer and M.~Meiners.
\newblock Fixed points of the smoothing transform: {T}wo-sided solutions.
\newblock {\em arXiv:1009.2412}, 1, 2010.

\bibitem{Biggins_77}
J.D. Biggins.
\newblock Martingale convergence in the branching random walk.
\newblock {\em Journal of Applied Probability}, 14(1):25--37, 1977.

\bibitem{Billingsley_1995}
P.~Billingsley.
\newblock {\em Probability and {M}easure}.
\newblock Wiley-Interscience, New York, 3rd edition, 1995.

\bibitem{BiGoTe1987}
N.H. Bingham, C.M. Goldie, and J.L. Teugels.
\newblock {\em Regular variation}.
\newblock Cambridge University Press, Cambridge, 1987.

\bibitem{Brandt_86}
A.~Brandt.
\newblock The stochastic equation $y_{n+1} = a_n y_n + b_n$ with stationary
  coefficients.
\newblock {\em Adv. Appl. Probab.}, 18(1):211--220, 1986.

\bibitem{Fill_Jan_01}
J.A. Fill and S.~Janson.
\newblock Approximating the limiting {Q}uicksort distribution.
\newblock {\em Random Structures Algorithms}, 19(3-4):376--406, 2001.

\bibitem{Goldie_91}
C.M. Goldie.
\newblock Implicit renewal theory and tails of solutions of random equations.
\newblock {\em Ann. Appl. Probab.}, 1(1):126--166, 1991.

\bibitem{Goldie_Grubel_96}
C.M. Goldie and R.~Gr\"ubel.
\newblock Perpetuities with thin tails.
\newblock {\em Adv. Appl. Prob.}, 28(2):463--480, 1996.

\bibitem{Grey_94}
D.R. Grey.
\newblock Regular variation in the tail behaviour of solutions of random
  difference equations.
\newblock {\em Ann. Appl. Probab.}, 4(1):169--183, 1994.

\bibitem{Iksanov_04}
A.M. Iksanov.
\newblock Elementary fixed points of the {BRW} smoothing transforms with
  infinite number of summands.
\newblock {\em Stochastic Process. Appl.}, 114:27--50, 2004.

\bibitem{Jel_Olv_10}
P.R. Jelenkovi\'c and M.~Olvera-Cravioto.
\newblock Information ranking and power laws on trees.
\newblock {\em Adv. Appl. Prob.}, 42(2):1057--1093, 2010.

\bibitem{Jel_Olv_11b}
P.R. Jelenkovi\'c and M.~Olvera-Cravioto.
\newblock Implicit renewal theorem for trees with general weights.
\newblock {\em arXiv:1012.2165}, 2011.

\bibitem{Jel_Olv_11a}
P.R. Jelenkovi\'c and M.~Olvera-Cravioto.
\newblock Implicit renewal theory on trees.
\newblock {\em arXiv:1006.3295}, 2011.

\bibitem{Jess_Miko_06}
A.H. Jessen and T.~Mikosch.
\newblock Regularly varying functions.
\newblock {\em Publications de L'Institut Mathematique, Nouvelle Serie},
  80(94):171--192, 2006.

\bibitem{Kesten_73}
H.~Kesten.
\newblock Random difference equations and renewal theory for products of random
  matrices.
\newblock {\em Acta Math.}, 131:207--248, 1973.

\bibitem{Kon_Mik_05}
D.G. Konstantinides and T.~Mikosch.
\newblock Large deviations and ruin probabilities for solutions to stochastic
  recurrence equations with heavy-tailed innovations.
\newblock {\em Ann. Probab.}, 33(5):1992--2035, 2005.

\bibitem{Liu_00}
Q.~Liu.
\newblock On generalized multiplicative cascades.
\newblock {\em Stochastic Process. Appl.}, 86:263--286, 2000.

\bibitem{Nei_Rus_04}
R.~Neininger and L.~R\"uschendorf.
\newblock A general limit theorem for recursive algorithms and combinatorial
  structures.
\newblock {\em Ann. Appl. Prob.}, 14(1):378--418, 2004.

\bibitem{Olvera_11}
M.~Olvera-Cravioto.
\newblock Asymptotics for weighted random sums.
\newblock {\em arXiv:1102.0301}, 2011.

\bibitem{Rosler_93}
U.~R\"{o}sler.
\newblock The weighted branching process.
\newblock {\em Dynamics of complex and irregular systems (Bielefeld, 1991)},
  pages 154--165, 1993.
\newblock Bielefeld Encounters in Mathematics and Physics VIII, World Science
  Publishing, River Edge, NJ.

\bibitem{Ros_Rus_01}
U.~R\"{o}sler and L.~R\"{u}schendorf.
\newblock The contraction method for recursive algorithms.
\newblock {\em Algorithmica}, 29(1-2):3--33, 2001.

\bibitem{Hof_Hoo_Van_05}
R.~van~der Hofstad, G.~Hooghiemstra, and P.~Van Mieghem.
\newblock Distances in random graphs with Þnite variance degrees.
\newblock {\em Random Structures and Algorithms}, 27(1):76--123, 2005.

\bibitem{Volk_Litv_10}
Y.~Volkovich and N.~Litvak.
\newblock Asymptotic analysis for personalized web search.
\newblock {\em Adv. Appl. Prob.}, 42(2):577--604, 2010.

\bibitem{Volk_Litv_Dona_07}
Y.~Volkovich, N.~Litvak, and D.~Donato.
\newblock Determining factors behind the pagerank log-log plot.
\newblock {\em Proceedings of the 5th {I}nternational {W}orkshop on
  {A}lgorithms and {M}odels for the {W}eb-{G}raph, {WAW} 2007}, 2007.

\end{thebibliography}

\end{document}